\documentclass{publicadapt}

\usepackage{amsthm}%

\allowdisplaybreaks%

{\theoremstyle{plain}%
 \newtheorem{theorem}{Theorem}

}

{\theoremstyle{remark}
\newtheorem{remark}{Remark}
}

{\theoremstyle{definition}

\newtheorem{example}{Example}
}

 \title{An iterative approach toward hypergeometric accelerations}

 \author{John M.\ Campbell}
 \address{Department of Mathematics and Statistics, Dalhousie University, Halifax, Nova Scotia, Canada}
 \email{jh241966@dal.ca}
 \thanks{The author was supported by a Killam Postdoctoral Fellowship from the Killam Trusts.}

 \date{May, 05, 2025}


\begin{document}
 \begin{abstract}
 Each of Ramanujan's series for $\frac{1}{\pi}$ is of the form $$ \sum_{n=0}^{\infty} z^n \frac{ (a_{1})_{n} (a_{2})_{n} (a_{3})_{n} }{ 
 (b_{1})_{n} (b_{2})_{n} (b_{3})_{n} } (c_{1} n + c_2) $$ for rational parameters such that the difference between the arguments of any lower 
 and upper Pochhammer symbols is not an integer. In accordance with the work of Chu, if an infinite sum of this form admits a closed 
 form, then this provides a formula of \emph{Ramanujan type}. Chu has introduced remarkable results on formulas of Ramanujan type, 
 through the use of accelerations based on $\Omega$-sums related to classical hypergeometric identities. Building on our past work on 
 an acceleration method due to Wilf relying on inhomogeneous difference equations derived from Zeilberger's algorithm, we extend this 
 method through what we refer to as an \emph{iterative} approach that is inspired by Chu's accelerations derived using iteration patterns 
 for well-poised $\Omega$-sums and that we apply to introduce and prove many accelerated formulas of Ramanujan type for 
 universal constants, along with many further accelerations related to the discoveries of Ramanujan, Guillera, and Chu. 
 \end{abstract}

 \maketitle

\section{Introduction}
 Ramanujan's series expansions for $\frac{1}{\pi}$ \cite[pp.\ 352--354]{Berndt1994} \cite{Ramanujan1914} are subjects of ongoing 
 fascination, and one of the most famous such expansions is such that 
\begin{equation}\label{Ramanujanbinomial}
 \frac{9801}{\pi \sqrt{8}} = \sum_{n=0}^{\infty} \frac{ (4n)! \left( 1103 + 26390 n \right) }{ (n!)^{4} 396^{4n} }. 
\end{equation}
 Gosper applied Ramanujan's formula in \eqref{Ramanujanbinomial} to compute a record-breaking number of digits of $\pi$, 
 with reference to the survey by Baruah et al.\ on Ramanujan's series \cite{BaruahBerndtChan2009}. Ramanujan's series for $\frac{1}{\pi}$ 
 are related to results due to Guillera 
 \cite{Guillera2018,Guillera2006,Guillera2008,Guillera2013More,Guillera2011,Guillera2010,Guillera2002,Guillera2013divergent} that 
 were discovered experimentally via the Wilf--Zeilberger (WZ) method \cite{PetkovsekWilfZeilberger1996} and that are often regarded as 
 groundbreaking in areas of mathematics related to WZ theory and Ramanujan's formulas for $\frac{1}{\pi}$. Such formulas are 
 similarly related to discoveries obtained by Chu and Zhang \cite{ChuZhang2014} via an experimental method that also relies on 
 difference equations for hypergeometric series. Our past work \cite{Campbellunpublished,CampbellLevrie2024free} concerning an 
 acceleration method due to Wilf \cite{Wilf1999} provides a foundation for the \emph{iterative} approach toward hypergeometric 
 accelerations that we describe in Section \ref{sectioniterative} below and that we apply to introduce many accelerated summations 
 that are for fundamental constants such as $\pi$ and that are related to and inspired by the discoveries due to Ramanujan, Guillera, 
 and Chu. 

 Let the \emph{Pochhammer symbol} or \emph{rising factorial} be defined so that $$(x)_{0} = 1 \ \ \ \text{and} \ \ \ (x)_{n} = x(x+1) 
 \cdots (x + n - 1)$$ for positive integers $n$, letting $x$ be referred to as the \emph{argument}. As in many research contributions by 
 Chu et al.\ related to hypergeometric recursions or accelerations, we adopt the notational shorthand such that 
\begin{equation*}
 \left[ \begin{matrix} \alpha, \beta, \ldots, \gamma \vspace{1mm} \\ 
 A, B, \ldots, C \end{matrix} \right]_{n} = \frac{ (\alpha)_{n} (\beta)_{n} 
 \cdots (\gamma)_{n} }{ (A)_{n} (B)_{n} \cdots (C)_{n}}. 
\end{equation*}
 With this notational convention, we may rewrite, for example, Ramanujan's series in 
 \eqref{Ramanujanbinomial} so that 
\begin{equation}\label{rewriteRamanujan}
 \frac{9801}{\pi \sqrt{8}} = \sum_{n=0}^{\infty} \left( \frac{1}{96059601} \right)^{n} 
 \left[ \begin{matrix} 
 \frac{1}{4}, \frac{1}{2}, \frac{3}{4} \vspace{1mm} \\ 
 1, 1, 1 \end{matrix} \right]_{n} (26390 n + 1103). 
\end{equation}
 This provides an instance of a formula that is said to be of \emph{Ramanujan type}, and a definition associated with the infinite series 
 involved in such formulas is given in Section \ref{sectiontype} below. Formulas of Ramanujan type provide a central object of study in this 
 work, which may be thought of as providing a 
 framework for applying Zeilberger's algorithm \cite[\S6]{PetkovsekWilfZeilberger1996} so as to obtain 
 series accelerations in a way that unifies 
 the techniques involved in a number of past research contributions 
 \cite{Campbellunpublished,CampbellLevrie2024free,Wilf1999}. 

\section{Formulas of Ramanujan type}\label{sectiontype}
 The formulas 
\begin{equation}\label{2over27Chu1}
 \frac{3\pi^2}{4} = 
 \sum_{n = 0}^{\infty} 
 \left( \frac{2}{27} \right)^{n} \left[ \begin{matrix} 
 1, 1, 1
 \vspace{1mm} \\ 
 \frac{4}{3}, \frac{5}{3}, \frac{3}{2} 
 \end{matrix} \right]_{n} (10n+7)
 \end{equation}
 and 
 \begin{equation}\label{2over27Chu2}
 \frac{3 \sqrt{2}}{4} = 
 \sum_{n = 0}^{\infty} 
 \left( \frac{2}{27} \right)^{n} \left[ \begin{matrix} 
 \frac{1}{2}, \frac{1}{2}, \frac{1}{2} 
 \vspace{1mm} \\ 
 \frac{5}{6}, 1, \frac{7}{6} 
 \end{matrix} \right]_{n} (5n+1) 
 \end{equation}
 are described as being of \emph{Ramanujan type} by Chu and are highlighted as main results in the work of Chu \cite{Chu2020} and 
 have inspired our new results highlighted in Section \ref{subsectionRamanujantype} below. The Chu formulas in \eqref{2over27Chu1} 
 and \eqref{2over27Chu2} recall Ramanujan's formula \cite[pp.\ 352--354]{Berndt1994} \cite{Ramanujan1914} 
\begin{equation}\label{Ramanujan2over27}
 \frac{27}{4 \pi} = 
 \sum_{n = 0}^{\infty} 
 \left( \frac{2}{27} \right)^{n} \left[ \begin{matrix} 
 \frac{1}{3}, \frac{1}{2}, \frac{2}{3} 
 \vspace{1mm} \\ 
 1, 1, 1 
 \end{matrix} \right]_{n} (15n+2) 
 \end{equation}
 involving the same convergence rate as in the Chu series in \eqref{2over27Chu1}--\eqref{2over27Chu2}. Observe that the Ramanujan 
 series among \eqref{rewriteRamanujan} and \eqref{Ramanujan2over27} and the series in the Chu formulas in 
 \eqref{2over27Chu1}--\eqref{2over27Chu2} are all of the form 
\begin{equation}\label{generalRT}
 \sum_{n=0}^{\infty} z^n \frac{ (a_{1})_{n} (a_{2})_{n} (a_{3})_{n} }{ (b_{1})_{n} (b_{2})_{n} (b_{3})_{n} } (c_{1} n + c_2) 
\end{equation} 
 for rational values $a_{i}$, $b_{i}$, $c_{i}$, and $z$ such that the difference between the arguments of any lower and upper 
 Pochhammer symbols is never an integer, i.e., so that any given combinations of rising factorials are irreducible, in terms of not being 
 reducible to rational functions. A series of this form that admits a closed-form evaluation is said to be of \emph{Ramanujan type}, 
 with reference to works by Chu \cite{Chu2018,Chu2020} and by Chu and Zhang \cite{ChuZhang2014}. 

 According to Chu and Zhang \cite{ChuZhang2014}, the first known formula of Ramanujan type is such that 
\begin{equation}\label{displayBauerRamanujan}
 \frac{2}{\pi} = 
 \sum_{n = 0}^{\infty} 
 (-1)^{n} \left[ \begin{matrix} 
 \frac{1}{2}, \frac{1}{2}, \frac{1}{2} 
 \vspace{1mm} \\ 
 1, 1, 1 
 \end{matrix} \right]_{n} (4n+1), 
 \end{equation}
 noting the absolute convergence rate of $1$ for the series displayed in \eqref{displayBauerRamanujan}, in contrast to the very fast 
 converging series among \eqref{rewriteRamanujan}--\eqref{Ramanujan2over27}. The rational expansion in 
 \eqref{displayBauerRamanujan} was first discovered in 1859 by Bauer \cite{Bauer1859} via a Fourier--Legendre expansion. The same 
 formula in \eqref{displayBauerRamanujan} was rediscovered by Ramanujan and was, famously, included in Ramanujan's first letter to 
 Hardy, and hence the expression \emph{Bauer--Ramanujan formula} in reference to the series expansion in 
 \eqref{displayBauerRamanujan} \cite{CampbellLevrie2024history}. There is a rich history surrounding the Bauer--Ramanujan formula 
 and its many different proofs 
 \cite{CampbellLevrie2024history}, motivating, from interdisciplinary perspectives and in view of the many different
 areas in mathematics associated with the known proofs of \eqref{displayBauerRamanujan}, 
 the investigation as to how 
 new techniques can be applied to generate formulas of Ramanujan type. 

 As suggested by the equivalence of \eqref{Ramanujanbinomial} and \eqref{rewriteRamanujan}, the combination of rising factorials 
 involved within the summand in \eqref{rewriteRamanujan} may be rewritten with binomial coefficients or factorials with integer-valued 
 arguments (and without any remaining Pochhammer symbols with non-integer arguments), according to the \emph{Gauss 
 multiplication formula} $$ \Gamma(n x) = \left( 2 \pi \right)^{\frac{1-n}{2}} n^{n x - \frac{1}{2}} \prod_{k=0}^{n-1} \Gamma\left( x + 
 \frac{k}{n} \right) $$ for the $\Gamma$-function such that $\Gamma(x) = \int_{0}^{\infty} t^{x-1} e^{-t} \, dt$ for $\Re(x) > 0$, 
 referring to Rainville's classic text \cite{Rainville1960} for related background material. Each of Ramanujan's series for $\frac{1}{\pi}$ can 
 be similarly rewritten with binomial coefficients, and in such a way so that the theory of elliptic functions can be used to prove and 
 generalize these formulas for $\frac{1}{\pi}$, by expressing the generating functions for the associated combinations of binomial 
 coefficients using the special functions known as the \emph{complete elliptic integrals}. This elliptic functions-based approach was 
 pioneered by the Borwein brothers \cite{BorweinBorwein1987}, who, famously, first proved all of Ramanujan's formulas
 for $\frac{1}{\pi}$. In closely related ways, the theories of elliptic functions to alternative bases associated with Ramanujan's series 
 were developed by Berndt, Bhargava, and Garvan \cite{BerndtBhargavaGarvan1995} and heavily rely on modular relations and theta 
 functions. In contrast, if a series of the form shown in \eqref{generalRT} is such that its Pochhammer symbols cannot be combined to 
 form integer-valued factorials (without any rising factorials with non-integer arguments), then, typically, such series cannot be 
 expressed, at least in any direct or meaningful way, using classical properties of special functions such as elliptic or theta functions. In 
 this case, we say that the combination of Pochhammer symbols involved is \emph{irregular}. 

 Chu has introduced formulas of Ramanujan type involving irregular combinations of Pochhammer symbols and involving a variety of 
 different convergence rates, and the remarkable nature of such results is given by how the generating functions-based approach 
 suggested above cannot be applied. For example, through a recursive argument related to the $\Omega$-sums that are due to Chu 
 and Zhang \cite{ChuZhang2014}, Chu \cite{Chu2020} proved that 
\begin{equation}\label{ChuRT512}
 21 \pi = \sum_{n = 0}^{\infty} 
 \left( -\frac{27}{512} \right)^{n} \left[ \begin{matrix} 
 1, \frac{2}{3}, \frac{4}{3} 
 \vspace{1mm} \\ 
 \frac{5}{4}, \frac{11}{8}, \frac{15}{8} 
 \end{matrix} \right]_{n} (77n+68). 
 \end{equation}
 Again with the use of recursions for $\Omega$-sums, Chu \cite{Chu2018} has proved formulas of Ramanujan type as in 
\begin{equation}\label{ChuRT28p15}
 \frac{21 \sqrt{2}}{2} = 
 \sum_{n = 0}^{\infty} 
 \left( -\frac{1}{27} \right)^{n} \left[ \begin{matrix} 
 \frac{1}{4}, \frac{3}{4}, \frac{3}{4} 
 \vspace{1mm} \\ 
 1, \frac{11}{12}, \frac{19}{12}
 \end{matrix} \right]_{n} (28n + 15)
 \end{equation}
 and 
\begin{equation}\label{ChuRT7p3}
 \frac{1036 \sqrt[3]{4} \pi }{999 \sqrt{3}} = 
 \sum_{n = 0}^{\infty} 
 \left( -\frac{1}{27} \right)^{n} \left[ \begin{matrix} 
 2, \frac{1}{3}, \frac{1}{6} 
 \vspace{1mm} \\ 
 \frac{10}{9}, \frac{13}{9}, \frac{16}{9} 
 \end{matrix} \right]_{n} (7n + 3), 
\end{equation}
 noting the highly irregular combinations of Pochhammer symbols involved in 
 \eqref{ChuRT28p15} and \eqref{ChuRT7p3}, with the formula 
\begin{equation}
 \frac{9 \sqrt{3}}{\sqrt[3]{4} \pi} = 
 \sum_{n = 0}^{\infty} 
 \left( -\frac{1}{27} \right)^{n} \left[ \begin{matrix} 
 \frac{1}{6}, \frac{1}{3}, \frac{2}{3} 
 \vspace{1mm} \\ 
 1, 1, 1 
 \end{matrix} \right]_{n} (21n + 2)
\end{equation}
 highlighted as a main result by Chu \cite{Chu2018} and having been described as being of \emph{Ramanujan type}. Through the use of 
 an acceleration method related to Dougall's summation theorem, Chen and Chu \cite{ChenChu2024}, in 2024, introduced and proved 
 formulas of Ramanujan type that are highlighted as main results in their work, including the formula 
\begin{equation}
 \frac{27 \sqrt{3} }{ \pi \sqrt[3]{32} } = 
 \sum_{n = 0}^{\infty} 
 \left( -\frac{1}{27} \right)^{n} \left[ \begin{matrix} 
 \frac{1}{3}, \frac{2}{3}, \frac{5}{6} 
 \vspace{1mm} \\ 
 1, 1, 1 
 \end{matrix} \right]_{n} (42 n +5). 
\end{equation}
 Chen and Chu \cite{ChenChu2021} also applied techniques relying on known $q$-series 
 to derive formulas of Ramanujan type such as 
 \begin{equation}\label{ChuRT10p9} 
 \frac{1024}{27(3 + \sqrt{3})} = 
 \sum_{n = 0}^{\infty} 
 \left( -\frac{1}{4} \right)^{n} \left[ \begin{matrix} 
 \frac{1}{2}, \frac{5}{4}, \frac{7}{4} 
 \vspace{1mm} \\ 
 \frac{4}{3}, \frac{5}{3}, 2 
 \end{matrix} \right]_{n} (10n+9). 
\end{equation}
 The discoveries due to Chu et al.\ that are reproduced via 
 \eqref{ChuRT512}--\eqref{ChuRT10p9} 
 motivate the new formulas of Ramanujan type that are higlighted in Section \ref{subsectionRamanujantype} 
 as main motivating results and that we prove via our iterative method. 

\subsection{Iterated accelerations and formulas of Ramanujan type}\label{subsectionRamanujantype}
  The arrangements of Pochhammer symbols appearing in our new formulas of Ramanujan  type are typically irregular, so that  
  classical methods (via the use of power series and the like)   cannot be applied in such cases.   Our previous method of \emph{shifted 
  indices} \cite{Campbellunpublished}  (cf.\ \cite{CampbellLevrie2024free}) was applied to obtain a number of new formulas of Ramanujan  
  type,   but our new and iterative method   appears to be much more versatile,   as evidenced by the many new formulas of Ramanujan  
 type   that we introduce below.  

 Our iterative approach has led us to discover and prove the relation such that 
 \begin{equation}\label{newRamanujantype1}
 3 \pi = \sum_{j = 0}^{\infty} \left( \frac{16}{27} \right)^{j} \left[ \begin{matrix} 
 \frac{1}{2}, 1, \frac{3}{2} 
 \vspace{1mm} \\ 
 \frac{5}{4}, \frac{4}{3}, \frac{5}{3} 
 \end{matrix} \right]_{j} 
 (11 j+4). 
 \end{equation} 
 To the best of our knowledge, the formula of Ramanujan type in   \eqref{newRamanujantype1} is new,  with regard to     previous  
  research related to Chu and Zhang's accelerations via $\Omega$-sums   
 \cite{ChenChu2021Guillera,ChenChu2021,ChenChu2024,Chu2023,Chu2018,Chu2020,Chu2021poised,Chu2021GouldHsu,Chu2019,Zhang2015}.   
  In view of Ramanujan's formula \cite[pp.\ 352--354]{Berndt1994} \cite{Ramanujan1914} 
\begin{equation}\label{27025707424210PM1A}
 \frac{8}{\pi} = 
 \sum_{k = 0}^{\infty} \left(-\frac{1}{4} \right)^{k} 
 \left[ \begin{matrix} 
 \frac{1}{4}, \frac{1}{2}, \frac{3}{4} \vspace{1mm} \\ 
 1, 1, 1 
 \end{matrix} \right]_{k} 
 \left( 20 k + 3 \right), 
\end{equation}
 the new formulas below that are of the same convergence rate and that 
 are of Ramanujan type may be seen as natural counterparts to \eqref{27025707424210PM1A}: 
\begin{align}
 8 \sqrt{2} & = 
 \sum_{j = 0}^{\infty} \left( -\frac{1}{4} \right)^{j} \left[ \begin{matrix} 
 \frac{5}{8}, \frac{3}{4}, \frac{9}{8} 
 \vspace{1mm} \\ 
 \frac{1}{2}, 1, \frac{3}{2} 
 \end{matrix} \right]_{j} (40 j+19), \label{newRT2} \\ 
 16 \sqrt{2} & = 
 \sum_{j = 0}^{\infty} \left( -\frac{1}{4} \right)^{j} \left[ \begin{matrix} 
 \frac{1}{8}, \frac{5}{8}, \frac{3}{4} 
 \vspace{1mm} \\ 
 1, \frac{3}{2}, \frac{3}{2}
 \end{matrix} \right]_{j} (40 j+23), \label{newRT3} \\ 
 \frac{8 \sqrt{2}}{5} 
 & = \sum_{j = 0}^{\infty} \left( -\frac{1}{4} \right)^{j} \left[ \begin{matrix} 
 \frac{5}{8}, \frac{3}{4}, \frac{9}{8} 
 \vspace{1mm} \\ 
 1, \frac{5}{4}, \frac{5}{4} 
 \end{matrix} \right]_{j} (8 j+3), \label{newRT4} \\ 
 \frac{96 \sqrt{2}}{7} & = 
 \sum_{j = 0}^{\infty} \left( -\frac{1}{4} \right)^{j} \left[ \begin{matrix} 
 \frac{5}{4}, \frac{11}{8}, \frac{15}{8} 
 \vspace{1mm} \\ 
 1, \frac{7}{4}, \frac{7}{4} 
 \end{matrix} \right]_{j} (40 j+33), \label{newRT5} \\ 
 \frac{416 \sqrt{2}}{81} & = 
 \sum_{j = 0}^{\infty} \left( -\frac{1}{4} \right)^{j} \left[ \begin{matrix} 
 \frac{7}{8}, \frac{5}{4}, \frac{11}{8} 
 \vspace{1mm} \\ 
 1, \frac{17}{12}, \frac{25}{12} 
 \end{matrix} \right]_{j} (8 j+9), \label{newRT6} \\ 
 \frac{56 \sqrt{2}}{3} & = 
 \sum_{j = 0}^{\infty} \left( -\frac{1}{4} \right)^{j} \left[ \begin{matrix} 
 \frac{1}{8}, \frac{5}{8}, \frac{3}{4} 
 \vspace{1mm} \\ 
 \frac{11}{12}, 1, \frac{19}{12} 
 \end{matrix} \right]_{j} (40 j+27), \label{newRT7} \\ 
 \frac{256 \sqrt{3}}{9} & = 
 \sum_{j = 0}^{\infty} \left( -\frac{1}{4} \right)^{j} \left[ \begin{matrix} 
 -\frac{1}{12}, \frac{5}{12}, \frac{5}{6} 
 \vspace{1mm} \\ 
 1, \frac{5}{3}, \frac{5}{3} 
 \end{matrix} \right]_{j} (60j+49), \label{newRT8} \\ 
 -\frac{\sqrt{3}}{2} & = 
 \sum_{j = 0}^{\infty} \left( -\frac{1}{4} \right)^{j} \left[ \begin{matrix} 
 \frac{1}{12}, \frac{1}{6}, \frac{7}{12} 
 \vspace{1mm} \\ 
 -\frac{2}{3}, \frac{1}{3}, 1 
 \end{matrix} \right]_{j} (60 j-1), \label{newRT9} \\ 
 16 \log (2) & = 
 \sum_{j = 0}^{\infty} \left( -\frac{1}{4} \right)^{j} \left[ \begin{matrix} 
 \frac{3}{4}, 1, \frac{5}{4} 
 \vspace{1mm} \\ 
 \frac{3}{2}, \frac{3}{2}, \frac{3}{2} 
 \end{matrix} \right]_{j} (20 j+13), \label{newRT10} \\ 
 \frac{44 \sqrt[3]{2}}{9} & = 
 \sum_{j = 0}^{\infty} \left( -\frac{1}{4} \right)^{j} \left[ \begin{matrix} 
 \frac{2}{3}, \frac{3}{4}, \frac{5}{4} 
 \vspace{1mm} \\ 
 \frac{5}{12}, 1, \frac{23}{12} 
 \end{matrix} \right]_{j} (10 j + 9), \label{newRT11} \\ 
 \frac{20 \sqrt[3]{2} }{3} & = 
 \sum_{j = 0}^{\infty} \left( -\frac{1}{4} \right)^{j} \left[ \begin{matrix} 
 \frac{1}{4}, \frac{2}{3}, \frac{3}{4} 
 \vspace{1mm} \\ 
 \frac{11}{12}, 1, \frac{17}{12} 
 \end{matrix} \right]_{j} (20j + 9). \label{newRT12} 
\end{align}
 If we consider our new formulas of Ramanujan type involving series of convergence rate $-\frac{1}{4}$ 
 in relation to Ramanujan's series of the same convergence rate in \eqref{27025707424210PM1A}, 
 and if we then consider the Ramanujan series of positive convergence rate $\frac{1}{4}$ 
 \cite[pp.\ 352--354]{Berndt1994} \cite{Ramanujan1914} such that 
\begin{equation}\label{Ramanujanposquarter}
 \frac{4}{\pi} = 
 \sum_{k = 0}^{\infty} \left( \frac{1}{4} \right)^{k} 
 \left[ \begin{matrix} 
 \frac{1}{2}, \frac{1}{2}, \frac{1}{2} \vspace{1mm} \\ 
 1, 1, 1 
 \end{matrix} \right]_{k} 
 (6k+1), 
\end{equation}
 this motivates how we have applied our iterative method
 to obtain the following formulas of Ramanujan type containing series of the same convergence rate 
 as that in \eqref{Ramanujanposquarter}. 
 The following results appear to be new: 
\begin{align}
 \frac{105 \pi }{32 \sqrt{3}} & = 
 \sum_{j = 0}^{\infty} \left( \frac{1}{4} \right)^{j} \left[ \begin{matrix} 
 \frac{5}{3}, 2, \frac{7}{3} 
 \vspace{1mm} \\ 
 \frac{11}{6}, \frac{13}{6}, \frac{5}{2} 
 \end{matrix} \right]_{j} 
 (3 j+4), \label{pquarterRT1} \\ 
 \frac{7\sqrt{2}}{8} & = 
 \sum_{j = 0}^{\infty} \left( \frac{1}{4} \right)^{j} \left[ \begin{matrix} 
 \frac{1}{2}, \frac{5}{4}, \frac{3}{2} 
 \vspace{1mm} \\ 
 1, \frac{11}{8}, \frac{15}{8} 
 \end{matrix} \right]_{j} (j + 1), \label{pquarterRT2} \\ 
 \frac{4301 \sqrt{3}}{256} & = 
 \sum_{j = 0}^{\infty} \left( \frac{1}{4} \right)^{j} \left[ \begin{matrix} 
 \frac{3}{2}, \frac{7}{3}, \frac{8}{3} 
 \vspace{1mm} \\ 
 1, \frac{29}{12}, \frac{35}{12} 
 \end{matrix} \right]_{j} 
 (9 j+16), \label{202505121101PM1A} \\ 
 \frac{935 \sqrt{3}}{128} & = 
 \sum_{j = 0}^{\infty} \left( \frac{1}{4} \right)^{j} \left[ \begin{matrix} 
 \frac{1}{2}, \frac{5}{3}, \frac{7}{3} 
 \vspace{1mm} \\ 
 1, \frac{23}{12}, \frac{29}{12}
 \end{matrix} \right]_{j} (9 j+10), \label{202505121118PM1A} \\ 
 \frac{4301 \sqrt{3}}{256} & = 
 \sum_{j = 0}^{\infty} \left( \frac{1}{4} \right)^{j} \left[ \begin{matrix} 
 \frac{3}{2}, \frac{7}{3}, \frac{8}{3} 
 \vspace{1mm} \\ 
 1, \frac{29}{12}, \frac{35}{12}
 \end{matrix} \right]_{j} 
 (9 j+16), \label{nRTsqrt39p16} \\ 
 \frac{1729 \sqrt{3}}{128} & = 
 \sum_{j = 0}^{\infty} \left( \frac{1}{4} \right)^{j} \left[ \begin{matrix} 
 \frac{3}{2}, \frac{5}{3}, \frac{7}{3} 
 \vspace{1mm} \\ 
 1, \frac{25}{12}, \frac{31}{12} 
 \end{matrix} \right]_{j} 
 (9 j + 14), \label{202505121131PM1A} \\
 \frac{27 \sqrt[3]{2} }{4} & = 
 \sum_{j = 0}^{\infty} \left( \frac{1}{4} \right)^{j} \left[ \begin{matrix} 
 \frac{2}{3}, \frac{5}{6}, \frac{7}{6} 
 \vspace{1mm} \\ 
 1, \frac{5}{4}, \frac{7}{4} 
 \end{matrix} \right]_{j} (9 j+7). \label{20250513745PM1AAM1A} 
\end{align}

 \ 

\noindent {\bf Further closed forms:} In regard to our applications of the iterative method in this paper, we mainly restrict our attention to 
 formulas for fundamental constants such as $\pi$ and $\sqrt{2}$, as in the motivating results listed above. However, in view of the 
 extent of research interest in the past work of Chu et al.\ on formulas of Ramanujan type, we also highlight the new rational and 
 algebraic closed-form evaluations that are listed below and that we have discovered and proved through our iterative method: 
\begin{align}
 \frac{21}{2} & = 
 \sum_{j = 0}^{\infty} \left( -\frac{1}{4} \right)^{j} \left[ \begin{matrix} 
 \frac{1}{3}, \frac{2}{3}, \frac{4}{3} 
 \vspace{1mm} \\ 
 \frac{13}{12}, \frac{19}{12}, 2 
 \end{matrix} \right]_{j} (15 j+11), \label{furtherclosed1} \\ 
 \frac{22}{3} & = 
 \sum_{j = 0}^{\infty} \left( -\frac{1}{4} \right)^{j} \left[ \begin{matrix} 
 -\frac{1}{3}, \frac{2}{3}, \frac{4}{3} 
 \vspace{1mm} \\ 
 1, \frac{17}{12}, \frac{23}{12} 
 \end{matrix} \right]_{j} (6 j+7), \label{furthernewRT2} \\ 
 7 \sqrt[3]{4} & = 
 \sum_{j = 0}^{\infty} \left( -\frac{1}{4} \right)^{j} \left[ \begin{matrix} 
 -\frac{1}{12}, \frac{1}{3}, \frac{5}{12} 
 \vspace{1mm} \\ 
 1, \frac{13}{12}, \frac{19}{12} 
 \end{matrix} \right]_{j} (60 j+11), \label{furthernewRT3} \\ 
 \frac{9 \sqrt[4]{8} }{5} & = 
 \sum_{j = 0}^{\infty} \left( -\frac{1}{4} \right)^{j} \left[ \begin{matrix} 
 -\frac{1}{16}, \frac{3}{8}, \frac{7}{16} 
 \vspace{1mm} \\ 
 1, \frac{17}{16}, \frac{25}{16} 
 \end{matrix} \right]_{j} (16 j+3), \label{furthernewRT4} \\
 \frac{91 \sqrt[3]{4}}{108} & = 
 \sum_{j = 0}^{\infty} \left( \frac{1}{4} \right)^{j} \left[ \begin{matrix} 
 \frac{5}{6}, \frac{4}{3}, \frac{3}{2} 
 \vspace{1mm} \\ 
 1, \frac{19}{12}, \frac{25}{12}
 \end{matrix} \right]_{j} (j+1), \label{20250513856} \\ 
 13 \sqrt[6]{32} & = 
 \sum_{j = 0}^{\infty} \left( -\frac{1}{4} \right)^{j} \left[ \begin{matrix} 
 -\frac{1}{24}, \frac{5}{12}, \frac{11}{24} 
 \vspace{1mm} \\ 
 1, \frac{25}{24}, \frac{37}{24} 
 \end{matrix} \right]_{j} (120 j+23), \label{20250420323PM1A} \\ 
 \frac{493 \sqrt[6]{2}}{432} & = 
 \sum_{j = 0}^{\infty} \left( \frac{1}{4} \right)^{j} \left[ \begin{matrix} 
 \frac{7}{12}, \frac{3}{2}, \frac{11}{6} 
 \vspace{1mm} \\ 
 1, \frac{41}{24}, \frac{53}{24} 
 \end{matrix} \right]_{j} 
 (j + 1), \label{20250513323AM1A} \\
 \frac{ 64 \sqrt[6]{108} }{3} & = 
 \sum_{j = 0}^{\infty} \left( -\frac{1}{4} \right)^{j} \left[ \begin{matrix} 
 \frac{1}{12}, \frac{7}{12}, \frac{5}{6} 
 \vspace{1mm} \\ 
 1, \frac{4}{3}, \frac{5}{3} 
 \end{matrix} \right]_{j} (60j+47), \label{fourth20250420} \\ 
 \frac{64 \sqrt[6]{432} }{27} & = 
 \sum_{j = 0}^{\infty} \left( -\frac{1}{4} \right)^{j} \left[ \begin{matrix} 
 \frac{3}{4}, \frac{5}{6}, \frac{5}{4} 
 \vspace{1mm} \\ 
 1, \frac{4}{3}, \frac{4}{3} 
 \end{matrix} \right]_{j} (20 j+9), \label{2025042529202242P2M2A} \\ 
 \frac{24}{\sqrt{2-\sqrt{2}}} & = 
 \sum_{j = 0}^{\infty} \left( -\frac{1}{4} \right)^{j} \left[ \begin{matrix} 
 \frac{3}{16}, \frac{5}{8}, \frac{11}{16} 
 \vspace{1mm} \\ 
 \frac{7}{8}, 1, \frac{11}{8} 
 \end{matrix} \right]_{j} (80 j+33), \label{20250508731AM1A} \\ 
 48 \sqrt{4-2 \sqrt{2}} & = 
 \sum_{j = 0}^{\infty} \left( -\frac{1}{4} \right)^{j} \left[ \begin{matrix} 
 \frac{15}{16}, \frac{9}{8}, \frac{23}{16} 
 \vspace{1mm} \\ 
 1, \frac{11}{8}, \frac{15}{8} 
 \end{matrix} \right]_{j} (80 j+69), \label{20250508747MSAMA1A} \\ 
 40 \sqrt{4-2 \sqrt{2}} & = 
 \sum_{j = 0}^{\infty} \left( -\frac{1}{4} \right)^{j} \left[ \begin{matrix} 
 \frac{9}{16}, \frac{7}{8}, \frac{17}{16} 
 \vspace{1mm} \\ 
 1, \frac{9}{8}, \frac{13}{8} 
 \end{matrix} \right]_{j} (80 j+51), \label{202354305038755A4M1A} \\ 
 \frac{48}{5} \sqrt{2 \sqrt{2}-2} & = 
 \sum_{j = 0}^{\infty} \left( -\frac{1}{4} \right)^{j} \left[ \begin{matrix} 
 \frac{3}{16}, \frac{11}{16}, \frac{3}{4} 
 \vspace{1mm} \\ 
 1, \frac{11}{8}, \frac{3}{2} 
 \end{matrix} \right]_{j} (16 j+9). \label{202501111508888088AM1A} 
\end{align}

\noindent Our new formulas of Ramanujan type for nested radicals as in 
 \eqref{20250508731AM1A}--\eqref{202501111508888088AM1A}
 may be seen as improving upon our past work 
 on series for nested radicals accelerated via the WZ method \cite{CampbellNested}, 
 since this past work did not introduce formulas of Ramanujan type. 
   The nested radical formulas among 
 \eqref{20250508731AM1A}--\eqref{202501111508888088AM1A} 
 may also be considered in relation to Ramanujan-inspired series due to Chu 
 involving nested radicals, as in the Chu series 
 \begin{equation*}
 \frac{32 \sqrt{2 - \sqrt{2}} }{\pi} = 
 \sum_{j = 0}^{\infty} \left( \frac{1}{4} \right)^{j} \left[ \begin{matrix} 
 \frac{1}{8}, \frac{1}{8}, \frac{7}{8}, \frac{7}{8} 
 \vspace{1mm} \\ 
 1, 1, 1, \frac{3}{2} 
 \end{matrix} \right]_{j} (192j^2 + 128 j + 7)
\end{equation*}
 introduced via Gould--Hsu relations \cite{Chu2021GouldHsu}, 
 and in relation to further research areas associated with nested radicals and $\pi$ formulas, 
 with reference to the work of Abrarov and Quine \cite{AbrarovQuine2018}. 

 \ 

\noindent {\bf Formulas of Bauer--Ramanujan type:} 
 Our iterative method may also be applied to obtain 
 formulas of Ramanujan type of the same convergence rate 
 as the Bauer--Ramanujan formula. 
 We refer to such formulas as being of \emph{Bauer--Ramanujan type}. 
 Despite the slower convergence rate of our new series in 
 \begin{equation}\label{notMapleBR0}
 \frac{8 \log (2)}{3} = 
 \sum_{j = 0}^{\infty} \left( -1 \right)^{j} \left[ \begin{matrix} 
 \frac{5}{6}, 1, \frac{7}{6} 
 \vspace{1mm} \\ 
 \frac{4}{3}, \frac{3}{2}, \frac{5}{3} 
 \end{matrix} \right]_{j} (4 j+3), 
\end{equation}
 the formula of Bauer--Ramanujan type in \eqref{notMapleBR0} appears to be new. 
 We have also discovered a way of applying our iterative method so as to obtain 
 formulas of Bauer--Ramanujan typer as in 
\begin{align}
 \frac{8 \sqrt{2}}{\pi } & = 
 \sum_{j = 0}^{\infty} \left( -1 \right)^{j} \left[ \begin{matrix} 
 \frac{1}{4}, \frac{5}{4}, \frac{5}{4} 
 \vspace{1mm} \\ 
 1, 1, 2 
 \end{matrix} \right]_{j} (8 j+5) \label{MapleBR1} 
\end{align}
 and 
\begin{align}
 \frac{8 \sqrt{2}}{\pi } = 
 \sum_{j = 0}^{\infty} \left( -1 \right)^{j} \left[ \begin{matrix} 
 -\frac{1}{4}, \frac{3}{4}, \frac{3}{4} 
 \vspace{1mm} \\ 
 1, 1, 2 
 \end{matrix} \right]_{j} (8 j+3), \label{MapleBR2} 
\end{align}
 but the Maple system can evaluate both of the series in \eqref{MapleBR1} and \eqref{MapleBR2}, 
 in contrast to how CAS software cannot evaluate any of the 
 series highlighted as Examples of our iteration method. 

 \ 

\noindent {\bf Recovered formulas of Ramanujan type:} Through our experimental use
 of the iterated acceleration method introduced in this paper, we have recovered the formulas 
 of Ramanujan type such that 
\begin{equation}\label{recover1}
 \frac{5 \pi }{8} = 
 \sum_{j = 0}^{\infty} \left( -\frac{1}{4} \right)^{j} \left[ \begin{matrix} 
 \frac{1}{4}, \frac{1}{4}, 1 
 \vspace{1mm} \\ 
 \frac{9}{8}, \frac{3}{2}, \frac{13}{8} 
 \end{matrix} \right]_{j} (5 j+2). 
\end{equation}
 and such that 
\begin{equation}\label{recoveranother}
 \frac{15 \pi }{4} = 
 \sum_{j = 0}^{\infty} \left( \frac{16}{27} \right)^{j} \left[ \begin{matrix} 
 \frac{1}{4}, \frac{3}{4}, 1 
 \vspace{1mm} \\ 
 \frac{1}{2}, \frac{7}{6}, \frac{11}{6} 
 \end{matrix} \right]_{j} (11 j+8), 
\end{equation}
 with \eqref{recover1} and \eqref{recoveranother} having been introduced
 and proved by Chu and Zhang \cite[Examples 7 and 47]{ChuZhang2014}
 via their acceleration method derived using Dougall's ${}_{5}F_{4}$-sum
 and an Abel-type summation formula. Moreover, 
 our iterated approach toward hypergeometric accelerations has led us to recover the formula 
\begin{equation}\label{third20250420}
 \frac{64 \sqrt{3}}{3} = 
 \sum_{j = 0}^{\infty} \left( -\frac{1}{4} \right)^{j} \left[ \begin{matrix} 
 \frac{5}{12}, \frac{5}{6}, \frac{11}{12}
 \vspace{1mm} \\ 
 \frac{2}{3}, 1, \frac{5}{3}
 \end{matrix} \right]_{j} (60j+43) 
\end{equation}
 of Ramanujan type introduced and proved by Chu 
 \cite[Example 92]{Chu2021poised} 
 through an acceleration method related to 
 what are referred to as very well-poised {$\Omega$}-sum. 
 This is representative of the versatility of our iterated acceleration method 
 and how it may be seen as ``unifying'' previously published results. 
 Our iterative method also allows us to construct a new proof of the formula 
 \begin{equation}\label{202540429935PM1A}
 \frac{3 \pi ^2}{2} = 
 \sum_{j = 0}^{\infty} \left( \frac{16}{27} \right)^{j} \left[ \begin{matrix} 
 1, 1, 1 
 \vspace{1mm} \\ 
 \frac{4}{3}, \frac{3}{2}, \frac{5}{3} 
 \end{matrix} \right]_{j} 
 (11 j+8). 
 \end{equation}
 proved by Zhang \cite{Zhang2015}
 and previously conjectured by Sun. 
 We also apply our iterative method to recover the formula 
\begin{equation}\label{FabryGuillera}
 \frac{\pi^2}{4} = 
 \sum_{j = 0}^{\infty} \left( \frac{1}{4} \right)^{j} \left[ \begin{matrix} 
 1, 1, 1 
 \vspace{1mm} \\ 
 \frac{3}{2}, \frac{3}{2}, \frac{3}2{} 
 \end{matrix} \right]_{j} (3j+2) 
 \end{equation}
 that was introduced by Fabry in 1911 \cite[p.\ 129]{Fabry1911} 
 and that was independently rediscovered by Guillera via the WZ method in 
 2008 \cite{Guillera2008}. 
 In view of the known derivations of $q$-analogues of equivalent versions of 
 the Fabry--Guillera formula in 
 \eqref{FabryGuillera} \cite{CampbellArchiv,ChenChu2021,HouKrattenthalerSun2019,WangZhong2023,Wei2020}, 
 the known $q$-analogues 
 of \eqref{FabryGuillera} 
 raise questions as to how our iterative method could be adapted so 
 as to obtain $q$-analogues for the remaining formulas of Ramanujan type given in this paper. 
 Through our iterative method, we have also discovered and proved the formula 
\begin{equation}\label{GuilleraGosper}
 \frac{9 \pi ^2}{16} = 
 \sum_{j = 0}^{\infty} \left( \frac{1}{4} \right)^{j} \left[ \begin{matrix} 
 1, 2, 3 
 \vspace{1mm} \\ 
 \frac{3}{2}, \frac{5}{2}, \frac{5}{2} 
 \end{matrix} \right]_{j} 
 (3 j+4) 
\end{equation}
 of Ramanujan type, which can be shown, via Gosper's algorithm, 
 to be eqiuvalent to the Fabry--Guillera formula. 
 Through a ternary version of our iterative method (described below), we have recovered the formula 
 \begin{equation}\label{202505191141AM1A}
 \frac{4\pi^2}{3} = 
 \sum_{j = 0}^{\infty} \left( \frac{1}{64} \right)^{j} \left[ \begin{matrix} 
 1, 1, 1 
 \vspace{1mm} \\ 
 \frac{3}{2}, \frac{3}{2}, \frac{3}{2} 
 \end{matrix} \right]_{j} 
 (21 j +13).
 \end{equation}
 of Ramanujan type proved by Guillera \cite{Guillera2008}. 

 We have may also applied our iterative method to prove formulas of Ramanujan type such as 
\begin{equation}\label{7202757075414347465837PM2A} 
 \frac{64}{3 \pi } = 
 \sum_{j = 0}^{\infty} \left( \frac{1}{4} \right)^{j} \left[ \begin{matrix} 
 \frac{1}{2}, \frac{3}{2}, \frac{5}{2} 
 \vspace{1mm} \\ 
 1, 2, 2 
 \end{matrix} \right]_{j} (6 j+5) 
 \end{equation}
 and 
 \begin{equation}\label{quasinewRT2p1}
 \frac{32}{3 \pi } = 
 \sum_{j = 0}^{\infty} \left( \frac{1}{4} \right)^{j} \left[ \begin{matrix} 
 \frac{3}{2}, \frac{3}{2}, \frac{3}{2} 
 \vspace{1mm} \\ 
 1, 1, 2 
 \end{matrix} \right]_{j} 
 (2j+1), 
 \end{equation}
 and these formulas can be shown, via Gosper's algorithm, 
 to be equivalent to Ramanujan's formula in \eqref{Ramanujanposquarter}, 
 and a simliar approach can be used to obtain equivalent versions of Ramanujan's series of convergence
 rate $-\frac{1}{4}$. 
 Such equivalences can also be established using 
 the evaluation for the generating function for 
 the sequence of cubed central binomial coefficients in terms of the complete elliptic integral of the first kind, 
 and by using the modular relations 
 underlying the derivation of Ramanujan's expansions for $\frac{1}{\pi}$, 
 again with reference to the classic monograph on $\pi$ and the AGM due to the Borwein brothers 
 \cite{BorweinBorwein1987}, and referring the interested reader to Cooper's text 
 on Ramanujan's theta functions \cite{Cooper2017} for related background material. 
 As expressed above, elliptic theory-based or theta function-based approaches cannot be applied 
 to series involving irregular combinations of Pochhammer symbols, and such series are a main 
 object of study in our work. 

\section{An iterative approach}\label{sectioniterative}
 In our previous work on two-term hypergeometric recurrences \cite{CampbellLevrie2024free}, 
 this concerned building upon an acceleration method due to Wilf \cite{Wilf1999} 
 by applying Zeilberger's algorithm 
 to bivariate hypergeometric functions involving multiple free parameters, such as 
\begin{equation}\label{207275707424733PM1A}
 F(n, k) := \left[ \begin{matrix} 
 a, b \vspace{1mm} \\ 
 n, n
 \end{matrix} \right]_{k}, 
\end{equation}
 in the hope that the inputting $F$-function would yield, according to Zeilberger's algorithm, 
 a two-term, inhomogeneous recurrence satisfying certain additional conditions. For the sake of brevity, 
 we omit a full discussion on the technical details and background material 
 concerning Wilf's method \cite{Wilf1999}, inviting the interested reader 
 to review our past works on this method \cite{Campbellunpublished,CampbellLevrie2024free}. 
 As a follow-up to our applications of Wilf's method using input functions as in \eqref{207275707424733PM1A}, 
 we then considered how this method could be further extended with the use of \emph{shifted indices} 
 of Pochhammer symbols \cite{Campbellunpublished}, 
 so as to be applicable to input functions such as 
\begin{equation}\label{27072775707570797775777P7M71A}
 F(n, k) := \frac{ (a)_{k+f} (b)_{k+e} }{ (n)_{k+d} (n)_{k + c} }. 
\end{equation}
 This method of shifted indices has been shown to be powerful, and has been applied 
 \cite{Campbellunpublished} to produce formulas of Ramanujan type of a similar nature
 relative to the new results highlighted in Section \ref{subsectionRamanujantype}. 
 This leads us to take Wilf's method even further, by extending the method of shifted indices
 using an \emph{iterative} approach described below. Our past results concerning Wilf's method are special
 cases of our new and iterative approach described below. 

\subsection{The iteration pattern (1, 1)}
 The iterative approach described as follows is inspired by 
 what Chu et al.\ have described as an \emph{iteration pattern} 
 given by expressing 
 $\Omega(a;b,c,d,e)$ in terms of 
 $\Omega(a+\rho_a; b + \rho_b, c + \rho_c, d + \rho+d, e + \rho_e)$, where, for 
 $\Re(1 + 2a - b - c - d - e) > 0$, we write 
 \begin{equation*} 
 \Omega(a;b,c,d,e) := 
 \sum_{k = 0}^{\infty} (a + 2 k) \left[ \begin{matrix} 
 b, c, d, e 
 \vspace{1mm} \\ 
 1 + a - b, 1 + a - c, 1 + a - d, 1 + a - e
 \end{matrix} \right]_{k}, 
\end{equation*}
 with this acceleration method having been pioneered
 in the above referenced Chu--Zhang 
 contribution Chu and Zhang \cite{ChuZhang2014} 
 on the acceleration of Dougall's {${}_5F_4$}-summation. 
 As expressed in our previous work \cite{Campbellunpublished}
 our acceleration techniques relying on Zeilberger's algorithm may be seen as more versatile, 
 in the sense that it does not rely on summands for $\Omega$-sums of the above form. 

 Let $ F(n, k, a) $ 
 be hypergeometric when interpreted as function of $n$ and $k$, for a fixed parameter $a$. 
 Suppose that Zeilberger's algorithm applied to $ F(n, k, a) $ 
 produces a two-term recurrence. For the sake of clarity, we let this recurrence be a first-order recurrence, 
 but our iterative method may be applied more broadly, writing
 $$ p_{1}(n, a) F(n + 1, k, a) + p_{2}(n, a) F(n, k, a) = G(n, k + 1, a) - G(n, k, a) $$
 for polynomials $p_{1}(n, a)$ and $p_{2}(n, a)$ in $n$ (again for a parameter $a$)
 and for a function $G(n, k, a)$ that is hypergeometric with respect to $n$ and $k$, 
 writing $G(n, k, a) = R(n, k, a) F(n, k, a)$ for a rational function $R(n, k, a)$ in $n$ and $k$. 
 We also assume that the input function $F(n, k, a)$ is such that 
 $\lim_{k \to \infty} G(n, k+1, a)$ vanishes (given specified bounds for $n$ and $a$). With this assumption, by writing 
\begin{equation}\label{202505036541PM1A}
 \mathfrak{f}(n, a) := \sum_{k=0}^{\infty} F(n, k, a), 
\end{equation}
 we find that a telescoping phenomenon 
 obtained from the above difference equation gives us that 
\begin{equation}\label{telescopep1p2}
 p_{1}(n, a) \, \mathfrak{f}(n+1, a) + p_{2}(n, a) \, \mathfrak{f}(n, a) = - G(n, 0, a). 
\end{equation}
 By setting $g_{1}(n, a) := -\frac{G(n, 0, a)}{p_{2}(n, a)}$
 and $g_{2}(n, a) := -\frac{p_{1}(n ,a)}{p_{2}(n, a)}$, we proceed to rewrite the difference equation in 
 \eqref{telescopep1p2} so that 
\begin{equation}\label{iteration1c0}
 \mathfrak{f}(n, a) = g_{1}(n, a) + g_{2}(n, a) \, \mathfrak{f}(n+1, a). 
\end{equation}
 We also suppose that $F(n, k, a)$, when viewed as bivariate function of $a$ and $k$, 
 is hypergeometric, letting $n$ be seen as a fixed parameter. We write 
\begin{equation}\label{mathcalFdefine}
 \mathcal{F}(n, k) := F(m, k, n)
\end{equation}
 for a parameter $m$. For the sake of clarity, we suppose that the application of 
 Zeilberger's algorithm to 
 $\mathcal{F}(n, k)$ produces a first-order recurrence, but, again, 
 our iterative method may be applied more broadly. 

 Let $\mathcal{F}(n, k)$ satisfy 
\begin{equation}\label{20250580838385838AMA1A}
 q_{3}(n, m) \, \mathcal{F}(n+1, k) + q_{4}(n, m) \, \mathcal{F}(n, k) 
 = \mathcal{H}(n, k+1) - \mathcal{H}(n, k) 
\end{equation}
 according to Zeilberger's algorithm, for polynomials $q_{3}(n, m)$ and $q_{4}(n ,m)$
 in $n$ and for a hypergeometric function $\mathcal{H}(n, k)$
 satisfying $\mathcal{H}(n$, $k)$ $=$ $ \mathcal{S}(n$, $k) $ $ \mathcal{F}(n$, $k)$ 
 for a rational function $\mathcal{S}(n, k)$. We then rewrite \eqref{20250580838385838AMA1A} so that 
\begin{equation*}
 q_{3}(n, m) \, F(m, k, n + 1) + q_{4}(n, m) \, F(m, k, n) 
 = \mathcal{H}(n, k+1) - \mathcal{H}(n, k), 
\end{equation*}
 and we write $\mathcal{H}(n, k) = \mathcal{H}(n, k, m)$ 
 and $\mathcal{S}(n, k) = \mathcal{S}(n, k, m)$. 
 We then write 
 $ p_{3}(n, a) = q_{3}(a, n)$ and 
 $ p_{4}(n, a) = q_{4}(a, n)$ and 
 $ \mathcal{G}(n, k, a) = \mathcal{H}(a, k, n)$, yielding 
\begin{equation}\label{20250509407PM1A}
 p_{3}(n, a) F(n, k, a + 1) + p_{4}(n, a) F(n, k, a) 
 = \mathcal{G}(n, k + 1, a) - \mathcal{G}(n, k, a), 
\end{equation}
 and we also write 
 $$ F(n, k, a) = \mathcal{R}(n, k, a) \mathcal{G}(n, k, a). $$
 By applying the summation operator $\sum_{k=0}^{\infty} \cdot$ to both sides, 
 we assume that the input function $F(n, k, a)$ is such that a telescoping phenomenon 
 in conjunction with a vanishing condition give us that 
\begin{equation}\label{2022502520236525PM2A}
 p_{3}(n, a) \, \mathfrak{f}(n, a+1) + p_{4}(n, a) \, \mathfrak{f}(n, a) 
 = - \mathcal{G}(n, 0, a), 
\end{equation}
 for $\mathfrak{f}$ as specified in \eqref{202505036541PM1A}. 
 By then setting $g_{3}(n, a) = -\frac{\mathcal{G}(n, 0, a)}{p_{4}(n, a)}$
 and $g_{4}(n, a) = -\frac{p_{3}(n, a)}{p_{4}(n, a)}$, we may obtain from \eqref{2022502520236525PM2A} that 
\begin{equation}\label{2120255505505536555855PM5A}
 \mathfrak{f}(n, a) 
 = g_{3}(n, a) + g_{4}(n, a) \, \mathfrak{f}(n, a+1) 
\end{equation}
 This leads us to apply the $\mathfrak{f}$-recursions among 
 \eqref{iteration1c0} and \eqref{2120255505505536555855PM5A} in an iterative manner, 
 i.e., by applying the recursion in \eqref{2120255505505536555855PM5A}
 using the right-hand side of \eqref{iteration1c0}. 
 Explicitly, by setting 
 $g_{5}(n, a) = g_{1}(n, a) + g_{2}(n, a) g_{3}(n+1, a)$
 and $g_{6}(n, a) = g_{2}(n, a) g_{4}(n+1, a)$, we find that 
\begin{equation}\label{202505067425PM2A}
 \mathfrak{f}(n, a) 
 = g_{5}(n, a) + g_{6}(n, a) \, \mathfrak{f}(n + 1, a + 1). 
\end{equation}
 Since we have expressed 
 $ \mathfrak{f}(n, a) $ in terms of 
 $ \mathfrak{f}(n + 1, a + 1)$ via the two-term recursion in 
 \eqref{202505067425PM2A}, we say that this recursion is of iteration pattern $(1, 1)$, 
 by analogy with the above referenced work of Chu concerning iteration patterns. 
 We are to apply recursions of the form indicated in \eqref{202505067425PM2A}, 
 as below, to obtain series accelerations, 
 and we later consider generalizations and variants of 
 our derivation of 
 \eqref{202505067425PM2A}, so as to obtain iterations of with patterns other than $(1, 1)$. 

 The repeated application of the recursion in \eqref{202505067425PM2A} yields 
\begin{multline*}
 \mathfrak{f}(n, a) = \sum_{j=0}^{m} \left( \prod_{i=0}^{j-1} g_{6}(n+i, a+i) \right) g_{5}(n+j, a+j) + \\ 
 \left( \prod_{i=0}^{m} g_{6}(n+i, a+i) \right) \mathfrak{f}(n+m+1, a+m+1) 
\end{multline*}
 for every positive integer $m$. We assume that the input function $F(n, k, a)$
 is such that the latter term in the above expansion for $ \mathfrak{f}(n, a) $
 vanishes after we set $m \to \infty$. The rate of acceleration on the right-hand side of 
\begin{equation}\label{20250507454PM1A}
 \mathfrak{f}(n, a) = \sum_{j = 0}^{\infty} \left( \prod_{i=0}^{j-1} g_{6}(n+i, a+i) \right) g_{5}(n+j, a+j) 
\end{equation}
 is determined by the coefficients involved in the $g_{6}$-expression, 
 and the identity in \eqref{20250507454PM1A} provides a key identity in regard to our iterative acceleration method 
 introduced and applied in this paper. 
 In addition to the new formulas of Ramanujan type that we have discovered 
 via \eqref{20250507454PM1A} and that are highlighted above, 
 new and motivating formulas for $\pi$ that are have discovered and proved
 through our iterative method are listed in Section 
 \ref{207777777727570747247273727AM1A} below. 

\subsection{Motivating $\pi$ formulas}\label{207777777727570747247273727AM1A}
 Since our current research contribution is based on extending 
 the method of shifted indices in our past work \cite{Campbellunpublished} (cf.\ \cite{CampbellLevrie2024free}), 
 we organize our results from our extended method in a compatible way, 
 with regard to the following listing of 
 new $\pi$ formulas highlighted as motivating results. 
 Each of the new $\pi$ formulas below 
 was obtained through the iterative method put forth in this paper, noting the irregular combinations 
 of Pochhammer symbols. For background related to 
 the importance of the development of $\pi$ formulas in the history of mathematics, 
 we refer to the classic monograph on $\pi$ and the AGM \cite{BorweinBorwein1987} 
 and further texts related to the history of $\pi$ formulas \cite{ArndtHaenel2001,Beckmann1971,BerggrenBorweinBorwein2004}. 
\begin{align}
 \frac{91 \pi }{18} & = 
 \sum_{j = 0}^{\infty} \left( \frac{1}{4} \right)^{j} \left[ \begin{matrix} 
 1, \frac{7}{6}, \frac{7}{6} 
 \vspace{1mm} \\ 
 \frac{5}{6}, \frac{19}{12}, \frac{25}{12} 
 \end{matrix} \right]_{j} 
 (18 j^2+25 j+6), \label{20727570757173310AM1A} \\
 \frac{280 \pi }{27 \sqrt{3}} & = 
 \sum_{j = 0}^{\infty} \left( -\frac{1}{4} \right)^{j} \left[ \begin{matrix} 
 1, \frac{3}{2} 
 \vspace{1mm} \\ 
 \frac{11}{6}, \frac{13}{6} 
 \end{matrix} \right]_{j} (30 j^2+58 j+27), \label{208828588084828984898P8M1A} \\ 
 \frac{1575 \pi }{32} & = 
 \sum_{j = 0}^{\infty} \left( \frac{16}{27} \right)^{j} \left[ \begin{matrix} 
 \frac{5}{4}, \frac{7}{4}, 2 
 \vspace{1mm} \\ 
 \frac{11}{6}, \frac{13}{6}, \frac{7}{2} 
 \end{matrix} \right]_{j} 
 (22 j^2+93 j+80), \label{2025043072226PM2A} \\ 
 \frac{21 \pi }{8} & = 
 \sum_{j = 0}^{\infty} \left( -\frac{1}{4} \right)^{j} \left[ \begin{matrix} 
 \frac{3}{4}, \frac{3}{4}, 1 
 \vspace{1mm} \\ 
 \frac{11}{8}, \frac{3}{2}, \frac{15}{8} 
 \end{matrix} \right]_{j} (20 j^2+29 j+10), \label{motia2025015128P2Mtin} \\ 
 -\frac{8 \pi }{3 \sqrt{3}} & = 
 \sum_{j = 0}^{\infty} \left( \frac{4}{27} \right)^{j} \left[ \begin{matrix} 
 -\frac{2}{3}, -\frac{1}{6} 
 \vspace{1mm} \\ 
 \frac{3}{2}, 1 
 \end{matrix} \right]_{j} 
 \frac{414 j^2+249 j-26}{(3 j+1) (6 j+5)}, \label{20250428911PM1A} \\ 
 -\frac{\pi }{8 \sqrt{2}} & = 
 \sum_{j = 0}^{\infty} \left( -\frac{1}{4} \right)^{j} \left[ \begin{matrix} 
 \frac{1}{4} 
 \vspace{1mm} \\ 
 \frac{1}{2} 
 \end{matrix} \right]_{j} \frac{160 j^3+92 j^2+5 j+1}{(4 j-1) (8 j+1) (8 j+5)}, \label{20250423motivating452four1} \\ 
 \frac{9 \pi ^2}{2} & = 
 \sum_{j = 0}^{\infty} \left( \frac{4}{27} \right)^{j} \left[ \begin{matrix} 
 \frac{1}{2}, 1, 1, 1 
 \vspace{1mm} \\ 
 \frac{5}{4}, \frac{4}{3}, \frac{5}{3}, \frac{7}{4} 
 \end{matrix} \right]_{j} 
 (92 j^2+123 j+40), \label{20250401333A111111M1A} \\
 \frac{55 \pi }{8 \sqrt{3}} & = 
 \sum_{j = 0}^{\infty} \left( -\frac{1}{4} \right)^{j} \left[ \begin{matrix} 
 \frac{5}{6}, \frac{5}{6}, 1, \frac{4}{3} 
 \vspace{1mm} \\ 
 \frac{7}{6}, \frac{17}{12}, \frac{3}{2}, \frac{23}{12} 
 \end{matrix} \right]_{j} (30 j^2+45 j+16), \label{20250423motivatingfive523} \\ 
 \frac{1024}{\pi } & = 
 \sum_{j = 0}^{\infty} \left( \frac{27}{64} \right)^{j} \left[ \begin{matrix} 
 \frac{3}{4}, \frac{5}{6}, \frac{7}{6}, \frac{5}{4} 
 \vspace{1mm} \\ 
 1, 1, \frac{3}{2}, 2 
 \end{matrix} \right]_{j} 
 (296 j^2+410 j+105), \label{20250429148AM1A} \\ 
 6 \pi ^2 & = 
 \sum_{j = 0}^{\infty} \left( \frac{4}{729} \right)^{j} \left[ \begin{matrix} 
 \frac{1}{2}, 1, 1, 1 
 \vspace{1mm} \\ 
 \frac{4}{3}, \frac{4}{3}, \frac{5}{3}, \frac{5}{3} 
 \end{matrix} \right]_{j} 
 (145 j^2+186 j+59), \label{20250651650146PM1A} \\ 
 \frac{5 \pi }{2} & = 
 \sum_{j = 0}^{\infty} \left( -\frac{1}{4} \right)^{j} \left[ \begin{matrix} 
 \frac{1}{2}, \frac{1}{2} 
 \vspace{1mm} \\ 
 \frac{9}{8}, \frac{13}{8} 
 \end{matrix} \right]_{j} \frac{80 j^3+168 j^2+115 j+26}{(j+1) (4 j+3)}, \label{20250323516motivating} \\ 
 \frac{6561 \sqrt{3}}{56 \pi } & = 
 \sum_{j = 0}^{\infty} \left( -\frac{1}{27} \right)^{j} \left[ \begin{matrix} 
 \frac{1}{6}, \frac{2}{3}, \frac{2}{3}, \frac{5}{3}, \frac{13}{6} 
 \vspace{1mm} \\ 
 1, \frac{3}{2}, \frac{3}{2}, 2, 2 
 \end{matrix} \right]_{j} (126 j^2+186 j+65). \label{motiv27bating20250423526} 
\end{align}

\subsection{Organization}
 Formulas of Ramanujan type may be regarded as ideal in terms of the (non-constant) polynomial factor 
 being minimal (i.e., linear). It is only in highly exceptional cases that 
 series satisfying the conditions associated with \eqref{generalRT} 
 admit closed forms, especially for irregular combinations of Pochhammer symbols, 
 and hence our emphasis on our new formulas of Ramanujan type. 
 This leads us to put a premium 
 on closed forms for series with polynomial factors that may be regarded as minimal, 
 and, hence, we restrict our attention to polynomial factors of degree not exceeding $3$. 
 Moreover, we disregard the cubic case, with the exception of 
 the case whereby 
 the number of Pochhammer symbols is minimal relative to formulas of Ramanujan type, 
 as in the new $\pi$ formulas 
 \begin{equation}\label{2707250757071727576AM1A}
 \frac{20 \pi }{\sqrt{3}} = 
 \sum_{j = 0}^{\infty} \left( -\frac{1}{4} \right)^{j} \left[ \begin{matrix} 
 \frac{2}{3} 
 \vspace{1mm} \\ 
 \frac{11}{6} 
 \end{matrix} \right]_{j} \frac{90 j^3+219 j^2+165 j+38}{(j+1) (2 j+1) (3 j+1)} 
\end{equation}
 and 
\begin{equation}\label{280828505808183018A8M8A}
 \frac{165 \pi}{8 \sqrt{3}} = 
 \sum_{j = 0}^{\infty} \left( -\frac{1}{4} \right)^{j} \left[ \begin{matrix} 
 \frac{1}{3}, \frac{5}{3} 
 \vspace{1mm} \\ 
 \frac{17}{12}, \frac{23}{12}
 \end{matrix} \right]_{j} \frac{90 j^3+231 j^2+174 j+38}{(j+1) (2 j+1) (6 j+1)} 
\end{equation} 
 that we highlight as applications of our iterative method, 
 and we also make an exception 
 for series for $\frac{1}{\pi}$ (which 
 seems to be rarer in ways that reflect the interest in Ramanujan's series for $\frac{1}{\pi}$), 
 as in the $\pi$ formula 
\begin{equation}\label{202504050130610000}
 \frac{192}{\pi } = 
 \sum_{j = 0}^{\infty} \left( \frac{4}{27} \right)^{j} \left[ \begin{matrix} 
 \frac{1}{4}, \frac{3}{4}, \frac{3}{4}, \frac{5}{4}, \frac{3}{2} 
 \vspace{1mm} \\ 
 1, 1, \frac{4}{3}, \frac{5}{3}, 2 
 \end{matrix} \right]_{j} 
 (368 j^3+704 j^2+378 j+45) 
\end{equation}
 introduced in this paper. 

 In view of the past literature referenced above on 
 fast converging $\pi$ formulas, equalities of this form are a 
 main subject of interest in this paper. As such, we put less of an emphasis on 
 formulas for constants not involving $\pi$, and we only include such formulas 
 if the series are such that the number of upper/lower Pochhammer symbols 
 is not exceeding that for formulas of Ramanujan type. 
 For example, the formula 
 \begin{equation}\label{20250513916AM1A}
 8 \sqrt[3]{2} = 
 \sum_{j = 0}^{\infty} \left( -\frac{1}{4} \right)^{j} \left[ \begin{matrix} 
 \frac{1}{3}, \frac{5}{6} 
 \vspace{1mm} \\ 
 1, \frac{5}{3} 
 \end{matrix} \right]_{j} (15 j + 11) 
\end{equation}
 satisfies the given conditions and may be considered in relation to the definition of a formula of Ramanujan type.

 Through extensive computer searches through hundreds of thousands of combinations 
 of rational parameter values inputted into our acceleration identities 
 highlighted as Theorems below, we have obtained 
 the 100 accelerated series that are highlighted as Examples
 below and that are for universal constants such as $\pi$, $\frac{1}{\pi}$, $\sqrt{2}$, and $\log 2$. 

\section{Accelerations yielding formulas of Ramanujan type}\label{sectionyielding}
 We begin by applying our iterative method
 in such a way so as to obtain the new formulas of Ramanujan type highlighted above. 
 To begin with, we consider what may be seen as a natural variant $F(n, k, a)$, as in \eqref{20250509423P7M7A}
 below, of the hypergeometric functions in 
 \eqref{207275707424733PM1A} and 
 \eqref{27072775707570797775777P7M71A}. 
 As we later demonstrate, Theorem \ref{firstmaintheorem} below can be used to 
 prove new formulas of Ramanujan type, as in 
 \eqref{newRT2} and \eqref{newRT3}. 

\begin{theorem}\label{firstmaintheorem}
 The acceleration identity in \eqref{20250507454PM1A} of iteration pattern $(1, 1)$ holds 
 for the input function 
\begin{equation}\label{20250509423P7M7A}
 F(n, k, a) := \frac{ (a)_{k + f} (b)_{k + e} }{ (n+a)_{k+d} (n)_{k+c} }. 
\end{equation}
\end{theorem}

\begin{proof}
 Applying Zeilberger's algorithm to $F(n, k) = F(n, k ,a)$ as specified, by treating $a$ as a fixed parameter, 
 we obtain a first-order recurrence of the desired form, with, according to the notation in 
 Section \ref{sectioniterative}, 
 the rational certificate 

 \ 

\noindent $ R(n, k, a) 
 = -n (a + n) \big( a d-a f+a n-b c-b d+b f-b k-2 b n+b+c^2+c d-c e-c f+c k+3 c n-c+d^2-d e-d f+d k+3 d n-d + 
 e f-e k-2 e n+e-f k-2 f n+f+2 k n+3 n^2-2 n \big)$. 

 \ 

\noindent and with polynomial coefficients 

 \ 

\noindent $ p_{1}(n, a) = (d-f+n) (a-c+f-n) (-b+c-e+n) (a-b+d-e+n) $ 

 \ 

\noindent and 

 \ 

\noindent $ p_{2}(n, a) = n (a+n) (b-c-d+e+f-2 n) (b-c-d+e+f-2 n+1). $

 \ 

\noindent Setting $\mathcal{F}(n, k)$ as specified in \eqref{mathcalFdefine}, we again obtain a first-order recurrence
 of the desired form, for a secondary rational certificate 

 \ 

\noindent $ \mathcal{R}(n, k, a) = (n+a) (c+k+n-1) $ 

 \ 

\noindent and for polynomials 

 \ 

\noindent $ p_{3}(n, a) = a (b-d+e-n-a) $ 

 \ 

\noindent and 

 \ 

\noindent $ p_{4}(n, a) = (n+a) (-c+f-n+a+1)$, 

 \ 

\noindent giving us that the first-order difference equation in \eqref{20250509407PM1A} 
 is satisfied, 
 for $p_{3}$ and $p_{4}$ and $\mathcal{G}$ as specified in 
 Section \ref{sectioniterative}. We thus obtain the desired recurrence 
 in \eqref{202505067425PM2A} of iteration pattern $(1, 1)$, 
 and we may thus verify that the desired acceleration identity 
\begin{equation}\label{2202225202520292421212PM2A}
 \mathfrak{f}(n, a) = \sum_{j = 0}^{\infty} \left( \prod_{i=0}^{j-1} g_{6}(n+i, a+i) \right) g_{5}(n+j, a+j) 
\end{equation}
 holds for $g_{5}$ and $g_{6}$ as specified in Section \ref{sectioniterative}
 and based on the input function in \eqref{20250509423P7M7A}, 
 again writing $ \mathfrak{f}(n, a) := \sum_{k=0}^{\infty} F(n, k, a)$. 
\end{proof}

 Inputting specific parameter values into the acceleration identity in 
 \eqref{2202225202520292421212PM2A} provides a versatile way of providing 
 accelerated series for universal constants, as demonstrated below. 
 Typically, if $ \mathfrak{f}(n, a)$ admits a closed form 
 for a given combination of rational values for the parameters 
 among the entries in the tuple $(a, b, c, d, e, f, n)$, 
 then it is a matter of routine to use classical techniques to prove 
 this closed form for $ \mathfrak{f}(n, a)$, since 
 $ \mathfrak{f}(n, a)$ is, up to a combination of Pochhammer symbols, a ${}_{3}F_{2}(1)$-hypergeometric series, 
 referring to Bailey's text \cite{Bailey1964} 
 on generalized hypergeometric series for background on 
 classical results on ${}_{3}F_{2}(1)$-identities. 
 So, we omit presenting derivations of the $ \mathfrak{f}(n, a)$-evaluations in 
 the below applications of Theorem \ref{firstmaintheorem}. 

\begin{example}
 Setting $(a, b, c, d, e, f, n) = 
 ( \frac{1}{3}, \frac{1}{3}, \frac{2}{3}, 0, -\frac{1}{2}, \frac{1}{2}, \frac{2}{3} )$ in Theorem \ref{firstmaintheorem}, we obtain 
 \begin{equation*} 
 \frac{64 \sqrt{3} }{9} = 
 \sum_{j = 0}^{\infty} \left( -\frac{1}{4} \right)^{j} \left[ \begin{matrix} 
 \frac{1}{6}, \frac{7}{12}, \frac{13}{12}
 \vspace{1mm} \\ 
 1, \frac{4}{3}, \frac{4}{3} 
 \end{matrix} \right]_{j} (120 j^2+94 j+15). 
\end{equation*}
\end{example}

\begin{example}
 Setting $(a, b, c, d, e, f, n) = 
 ( \frac{1}{6}, \frac{1}{6}, \frac{3}{4}, \frac{1}{4}, 0, 0, \frac{1}{4} )$ in Theorem \ref{firstmaintheorem}, we obtain 
 \begin{equation*} 
 \frac{128 \sqrt{3}}{9} = 
 \sum_{j = 0}^{\infty} \left( -\frac{1}{4} \right)^{j} \left[ \begin{matrix} 
 \frac{1}{4}, \frac{1}{2}, \frac{3}{4} 
 \vspace{1mm} \\ 
 1, \frac{4}{3}, \frac{5}{3} 
 \end{matrix} \right]_{j} (120 j^2+118 j+27) . 
\end{equation*}
\end{example}

\begin{example}
 Setting $(a, b, c, d, e, f, n) = 
 ( -\frac{1}{12}, -\frac{1}{12}, -\frac{1}{12}, \frac{1}{6}, -\frac{1}{12}, -\frac{1}{12}, \frac{11}{12} )$ in Theorem \ref{firstmaintheorem}, we obtain 
 \begin{equation*} 
 -\frac{\pi }{3} = 
 \sum_{j = 0}^{\infty} \left( -\frac{1}{4} \right)^{j} \left[ \begin{matrix} 
 \frac{7}{6} 
 \vspace{1mm} \\ 
 \frac{3}{2} 
 \end{matrix} \right]_{j} \frac{360 j^3+726 j^2+439 j+85}{(6 j-1) (12 j+7) (12 j+13)}. 
\end{equation*}
\end{example}

\begin{example}
 Setting $(a, b, c, d, e, f, n) = 
 ( \frac{1}{2}, \frac{1}{2}, \frac{1}{2}, \frac{3}{2}, -\frac{1}{4}, \frac{1}{4}, \frac{1}{2} )$ in Theorem \ref{firstmaintheorem}, we obtain 
 \begin{equation*} 
 160 \sqrt{2} = 
 \sum_{j = 0}^{\infty} \left( -\frac{1}{4} \right)^{j} \left[ \begin{matrix} 
 \frac{3}{4}, \frac{9}{8}, \frac{13}{8} 
 \vspace{1mm} \\ 
 1, \frac{9}{4}, \frac{9}{4} 
 \end{matrix} \right]_{j} (160 j^2+420 j+273). 
\end{equation*}
\end{example}

\begin{example}
 Setting $(a, b, c, d, e, f, n) = 
 ( \frac{1}{3}, \frac{1}{3}, \frac{2}{3}, 0, -\frac{1}{2}, \frac{1}{3}, \frac{2}{3} )$ in Theorem \ref{firstmaintheorem}, we obtain 
 \begin{equation*} 
 \frac{110 \sqrt[3]{2}}{3} = 
 \sum_{j = 0}^{\infty} \left( -\frac{1}{4} \right)^{j} \left[ \begin{matrix} 
 \frac{7}{12}, \frac{2}{3}, \frac{13}{12} 
 \vspace{1mm} \\ 
 1, \frac{17}{12}, \frac{23}{12} 
 \end{matrix} \right]_{j} (120 j^2+170 j+57).
\end{equation*}
\end{example}

\begin{example}
 Setting $(a, b, c, d, e, f, n) = 
 ( -\frac{1}{2}, -\frac{1}{2}, \frac{1}{4}, -\frac{1}{4}, \frac{1}{3}, -\frac{1}{3}, \frac{3}{4} )$ in Theorem \ref{firstmaintheorem}, we obtain 
\begin{align*}
 \frac{216}{\pi } & = 
 \sum_{j = 0}^{\infty} \left( -\frac{1}{4} \right)^{j} \left[ \begin{matrix} 
 -\frac{5}{6}, \frac{1}{12}, \frac{7}{12}, \frac{5}{6}, \frac{7}{6} 
 \vspace{1mm} \\ 
 \frac{1}{2}, 1, 1, \frac{3}{2}, 2 
 \end{matrix} \right]_{j} (2160 j^3+2556 j^2+552 j+35). 
\end{align*}
\end{example}

\begin{example}
 Setting $(a, b, c, d, e, f, n) = 
 ( \frac{1}{3}, \frac{2}{3}, \frac{1}{3}, 0, -\frac{1}{2}, -\frac{1}{2}, \frac{2}{3} )$ in Theorem \ref{firstmaintheorem}, we obtain 
\begin{equation*}
 \frac{1296}{\pi } = 
 \sum_{j = 0}^{\infty} \left( -\frac{1}{4} \right)^{j} \left[ \begin{matrix} 
 -\frac{1}{6}, \frac{5}{12}, \frac{5}{6}, \frac{11}{12}, \frac{7}{6} 
 \vspace{1mm} \\ 
 1, 1, \frac{3}{2}, \frac{3}{2}, 2 
 \end{matrix} \right]_{j} (2160 j^3+4140 j^2+2316 j+385). 
\end{equation*}
\end{example}

\begin{example}
 Setting $(a, b, c, d, e, f, n) = 
 ( \frac{1}{2}, \frac{1}{2}, -\frac{1}{2}, 0, \frac{1}{4}, -\frac{1}{4}, \frac{3}{2} )$ in Theorem \ref{firstmaintheorem}, we obtain 
 \begin{equation*} 
 \frac{2048 \sqrt{2}}{3 \pi } = 
 \sum_{j = 0}^{\infty} \left( -\frac{1}{4} \right)^{j} \left[ \begin{matrix} 
 \frac{1}{4}, \frac{1}{4}, \frac{5}{8}, \frac{9}{8}, \frac{7}{4} 
 \vspace{1mm} \\ 
 1, \frac{3}{2}, \frac{3}{2}, 2, 2 
 \end{matrix} \right]_{j} (640 j^3+1712 j^2+1360 j+315). 
\end{equation*}
\end{example}

\begin{example}
 Setting $(a, b, c, d, e, f, n) = 
 ( \frac{1}{2}, \frac{1}{2}, -\frac{1}{2}, 0, -\frac{1}{4}, \frac{1}{4}, \frac{3}{2} )$ in Theorem \ref{firstmaintheorem}, we obtain 
 \begin{equation*} 
 \frac{2048 \sqrt{2}}{3 \pi } = 
 \sum_{j = 0}^{\infty} \left( -\frac{1}{4} \right)^{j} \left[ \begin{matrix} 
 \frac{3}{4}, \frac{3}{4}, \frac{7}{8}, \frac{5}{4}, \frac{11}{8} 
 \vspace{1mm} \\ 
 1, \frac{3}{2}, \frac{3}{2}, 2, 2 
 \end{matrix} \right]_{j} (640 j^3+1680 j^2+1424 j+385). 
\end{equation*}
\end{example}

\begin{example}
 Setting $(a, b, c, d, e, f, n) = 
 ( -\frac{1}{2}, -\frac{1}{2}, -\frac{1}{4}, \frac{3}{4}, \frac{1}{4}, \frac{3}{4}, \frac{3}{4} )$ in Theorem \ref{firstmaintheorem}, we obtain 
 the formula of Ramanujan type highlighted in \eqref{newRT2}. 
\end{example}

\begin{example}
 Setting $(a, b, c, d, e, f, n) = 
 ( -\frac{1}{2}, \frac{1}{2}, \frac{3}{4}, \frac{3}{4}, \frac{1}{4}, \frac{3}{4}, \frac{3}{4} )$ in Theorem \ref{firstmaintheorem}, we obtain 
 the formula of Ramanujan type highlighted in \eqref{newRT3}.
\end{example}

\begin{example}
 Setting $(a, b, c, d, e, f, n) = 
 ( \frac{1}{2}, \frac{1}{2}, \frac{1}{2}, \frac{1}{2}, -\frac{1}{4}, \frac{1}{4}, \frac{1}{2} )$ in Theorem \ref{firstmaintheorem}, we obtain 
 the formula of Ramanujan type highlighted in \eqref{newRT4}. 
\end{example}

\begin{example}
 Setting $(a, b, c, d, e, f, n) = 
 ( -\frac{1}{12}, -\frac{1}{12}, 0, \frac{11}{12}, \frac{1}{6}, \frac{2}{3}, 1 )$ in Theorem \ref{firstmaintheorem}, we obtain 
 the formula of Ramanujan type highlighted in \eqref{newRT6}. 
\end{example}

\begin{example}
 Setting $(a, b, c, d, e, f, n) = 
 ( -\frac{1}{12}, -\frac{1}{12}, 0, -\frac{1}{12}, \frac{2}{3}, \frac{1}{6}, 1 )$ in Theorem \ref{firstmaintheorem}, we obtain 
 the formula of Ramanujan type highlighted in \eqref{newRT7}. 
\end{example}

\begin{example}
 Setting $(a, b, c, d, e, f, n) = 
 ( \frac{1}{6}, \frac{5}{6}, \frac{5}{6}, 0, \frac{1}{3}, 0, \frac{5}{6} )$ in Theorem \ref{firstmaintheorem}, we obtain 
 the formula of Ramanujan type highlighted in \eqref{newRT8}. 
\end{example}

\begin{example}
 Setting $(a, b, c, d, e, f, n) = 
 ( \frac{1}{3}, \frac{1}{3}, \frac{1}{3}, \frac{5}{6}, 0, \frac{5}{6}, \frac{2}{3} )$ in Theorem \ref{firstmaintheorem}, we obtain 
 the formula of Ramanujan type in \eqref{newRT11}. 
\end{example}

\begin{example}
 Setting $(a, b, c, d, e, f, n) = \big( \frac{1}{6}, \frac{1}{6}, \frac{1}{3}, 0, \frac{1}{6}, 0, \frac{2}{3} \big)$ 
 in Theorem \ref{firstmaintheorem}, 
 we obtain the formula of Ramanujan type in \eqref{newRT12}. 
\end{example}

\begin{example}
 Setting $(a, b, c, d, e, f, n) = 
 ( \frac{1}{6}, \frac{5}{6}, \frac{5}{6}, 0, \frac{1}{3}, \frac{1}{6}, \frac{5}{6} )$ in Theorem \ref{firstmaintheorem}, we obtain 
 the formula of Ramanujan type shown in \eqref{furthernewRT3}. 
\end{example}

\begin{example}
 Setting $(a, b, c, d, e, f, n) = 
 ( \frac{1}{6}, \frac{5}{6}, \frac{3}{4}, 0, \frac{1}{4}, \frac{1}{4}, \frac{5}{6} )$ in Theorem \ref{firstmaintheorem}, we obtain 
 the formula of Ramanujan type highlighted in \eqref{20250420323PM1A}. 
\end{example}

 \begin{example}
 Setting $(a, b, c, d, e, f, n) = 
 ( \frac{1}{6}, \frac{1}{6}, \frac{2}{3}, \frac{1}{6}, \frac{2}{3}, 0, \frac{2}{3} )$ in Theorem \ref{firstmaintheorem}, we obtain 
 the formula of Ramanujan type shown in \eqref{fourth20250420}. 
\end{example}

\begin{example}
 Setting $(a, b, c, d, e, f, n) = 
 ( -\frac{1}{8}, -\frac{1}{8}, 0, -\frac{1}{8}, \frac{1}{2}, \frac{1}{4}, 1 )$ in Theorem \ref{firstmaintheorem}, we obtain 
 the formula of Ramanujan type highlighted in \eqref{20250508731AM1A}. 
\end{example}

\begin{example}
 Setting $(a, b, c, d, e, f, n) = 
 ( -\frac{1}{8}, -\frac{1}{8}, 0, \frac{7}{8}, 0, \frac{3}{4}, 1 )$ in Theorem \ref{firstmaintheorem}, we obtain 
 the formula of Ramanujan type highlighted in \eqref{20250508747MSAMA1A}.
\end{example}

\begin{example}
 Setting $(a, b, c, d, e, f, n) = 
 ( -\frac{1}{8}, -\frac{1}{8}, 0, \frac{7}{8}, \frac{3}{4}, 1, 1 )$ in Theorem \ref{firstmaintheorem}, we obtain 
 the formula of Ramanujan type highlighted in \eqref{202354305038755A4M1A}. 
\end{example}

\begin{example}
 Setting $(a, b, c, d, e, f, n) = 
 ( -\frac{1}{8}, -\frac{1}{8}, \frac{3}{8}, \frac{1}{8}, \frac{3}{4}, \frac{3}{8}, 1 )$ in Theorem \ref{firstmaintheorem}, we obtain 
 the formula of Ramanujan type highlighted in \eqref{202501111508888088AM1A}. 
\end{example}

\begin{example}
 Setting $(a, b, c, d, e, f, n) = 
 ( \frac{1}{2}, \frac{1}{2}, \frac{3}{2}, \frac{1}{2}, \frac{1}{4}, -\frac{1}{4}, \frac{1}{4} )$ in Theorem \ref{firstmaintheorem}, we recover the formula
 in \eqref{recover1}
 of Ramanujan type introduced and proved by Chu and Zhang \cite[Example 7]{ChuZhang2014}. 
\end{example}

\begin{example}
 Setting $(a, b, c, d, e, f, n) = 
 ( \frac{1}{6}, \frac{1}{6}, 0, \frac{1}{6}, 0, 0, \frac{2}{3} )$ in Theorem \ref{firstmaintheorem}, we recover the formula in \eqref{third20250420} 
 of Ramanujan type introduced and proved by Chu \cite[Example 92]{Chu2021poised}. 
\end{example}

\begin{example}
 Setting $(a, b, c, d, e, f, n) = 
 ( \frac{1}{2}, \frac{1}{2}, \frac{3}{2}, \frac{3}{2}, -\frac{1}{4}, \frac{1}{4}, -\frac{1}{4} )$ in Theorem \ref{firstmaintheorem}, we obtain 
 the motivating $\pi$ formula shown in \eqref{motia2025015128P2Mtin}. 
\end{example}

\begin{example}
 Setting $(a, b, c, d, e, f, n) = 
 ( \frac{1}{6}, \frac{1}{6}, \frac{1}{3}, \frac{5}{6}, 0, \frac{2}{3}, \frac{5}{6} )$ in Theorem \ref{firstmaintheorem}, we obtain 
 the motivating $\pi$ formula shown in \eqref{20250423motivatingfive523}. 
\end{example}

\begin{example}
 Setting $(a, b, c, d, e, f, n) = 
 ( \frac{1}{2}, \frac{1}{2}, 1, \frac{3}{4}, \frac{1}{2}, 1, \frac{3}{4} )$ in Theorem \ref{firstmaintheorem}, we obtain 
 the motivating $\pi$ formula shown in \eqref{20250323516motivating}. 
\end{example}

\begin{example}
 Setting $(a, b, c, d, e, f, n) = 
 ( \frac{1}{2}, \frac{1}{2}, \frac{1}{2}, \frac{2}{3}, \frac{1}{2}, \frac{1}{6}, \frac{5}{6} )$ in Theorem \ref{firstmaintheorem}, we obtain 
 the motivating $\pi$ formula highlighted in \eqref{2707250757071727576AM1A}. 
\end{example}

\begin{example}
 Setting $(a, b, c, d, e, f, n) = 
 ( \frac{1}{3}, \frac{1}{3}, \frac{1}{3}, \frac{5}{6}, \frac{2}{3}, 0, \frac{5}{6} )$ in Theorem \ref{firstmaintheorem}, we obtain 
 the motivating $\pi$ formula highlighted in \eqref{280828505808183018A8M8A}. 
\end{example}

\begin{example}\label{cube15j}
 Setting $(a, b, c, d, e, f, n) = 
 ( \frac{1}{6}, \frac{1}{6}, \frac{1}{6}, \frac{1}{6}, \frac{1}{6}, 0, \frac{2}{3} )$ in Theorem \ref{firstmaintheorem}, 
 we obtain the motivating result highlighted in \eqref{20250513916AM1A}. 
\end{example}

\begin{remark}
 Apart from our recovered formulas of Ramanujan type, 
 our iterative method also allows us to recover many further accelerated 
 series from the work of Chu et al., 
 but, for the most part, we have neglected to include such recovered series, for the sake of brevity.
 For example, 
 setting $(a$, $b$, $c$, $d$, $e$, $f$, $n) $ $ = $
 $ ( \frac{1}{6}$, 
 $ \frac{1}{6}$, $ \frac{2}{3}$, $ \frac{5}{6}$, 
 $ \frac{1}{6}$, 
 $ \frac{1}{6}$, 
 $ \frac{1}{3} )$ in Theorem \ref{firstmaintheorem}, we 
 have independently rediscovered the formula 
 \begin{equation*} 
 \frac{40 \pi }{9 \sqrt{3}} = 
 \sum_{j = 0}^{\infty} \left( -\frac{1}{4} \right)^{j} \left[ \begin{matrix} 
 \frac{1}{2}, 1 
 \vspace{1mm} \\ 
 \frac{7}{6}, \frac{11}{6}
 \end{matrix} \right]_{j} (10 j+9) 
\end{equation*}
 that was previously given by Chu \cite[Example 21]{Chu2021poised}, 
 and, by setting 
 $(a$, $b$, $c$, $d$, $e$, $f$, $n) $ $ = $
 $ ( -\frac{1}{8}$, 
 $ -\frac{1}{8}$, 
 $ \frac{7}{8}$, 
 $ \frac{1}{4}$,
 $ \frac{7}{8}$, 
 $ \frac{7}{8}$, 
 $ \frac{7}{8} )$ in Theorem \ref{firstmaintheorem}, we 
 have independently rediscovered the formula 
 \begin{equation*} 
 \frac{\pi }{2 \sqrt{2}} = 
 \sum_{j = 0}^{\infty} \left( -\frac{1}{4} \right)^{j} \left[ \begin{matrix} 
 \frac{1}{4} 
 \vspace{1mm} \\ 
 \frac{3}{2} 
 \end{matrix} \right]_{j} 
 \frac{160 j^3+220 j^2+101 j+17}{(4 j+3) (8 j+1) (8 j+5)}
\end{equation*}
 given in an equivalent way by Chu 
 \cite[(B19)]{Chu2021GouldHsu}. 
\end{remark}

 For the same input function as in \eqref{20250509423P7M7A}, 
 by again deriving the recursions in \eqref{iteration1c0}
 and \eqref{2120255505505536555855PM5A}, 
 and by then applying this latter recursion twice so as to obtain a recursion of the form 
\begin{equation}\label{2120qq2q55q5q0q5q5qq0q5q5q3q6q5q5q5q8q5q5qPqM5qA}
 \mathfrak{f}(n, a) 
 = \overline{g_{3}}(n, a) + \overline{g_{4}}(n, a) \, \mathfrak{f}(n, a+2) 
\end{equation}
 for modified polynomial coefficients 
 $ \overline{g_{3}}$ and $\overline{g_{4}}$, 
 by then applying \eqref{iteration1c0} 
 in conjunction with \eqref{2120qq2q55q5q0q5q5qq0q5q5q3q6q5q5q5q8q5q5qPqM5qA}, 
 we obtain a recurrence of the form
\begin{equation}\label{202q5q0q5q0qqq6q7q4q2qqq5qPqMqq2A}
 \mathfrak{f}(n, a) 
 = \overline{g_{5}}(n, a) + \overline{g_{6}}(n, a) \, \mathfrak{f}(n + 1, a + 2), 
\end{equation}
 leaving the details to the reader. 
 The recurrence in \eqref{202q5q0q5q0qqq6q7q4q2qqq5qPqMqq2A} is said to be 
 of iteration pattern $(1, 2)$, by analogy with our derivation of 
 the recurrence in \eqref{202505067425PM2A} of iteration pattern $(1, 1)$. 
 Through the repeated application of \eqref{202q5q0q5q0qqq6q7q4q2qqq5qPqMqq2A} 
 to yield an analogue 
\begin{equation}\label{220w2w2w252w02w5w2w0w2w9w2w42w1w2w1w2wPwMw2A}
 \mathfrak{f}(n, a) = \sum_{j = 0}^{\infty} \left( \prod_{i=0}^{j-1} \overline{g_{6}}(n + i, a + 2i) \right) \overline{g_{5}}(n + j, a + 2j) 
\end{equation}
 of the acceleration in \eqref{2202225202520292421212PM2A}. 
 The identity in \eqref{220w2w2w252w02w5w2w0w2w9w2w42w1w2w1w2wPwMw2A}
 has led us to discover a new formula of Bauer--Ramanujan type, as below. 

\begin{example}
 Setting $(a, b, c, d, e, f, n) = 
 ( \frac{1}{2}, \frac{1}{2}, \frac{1}{2},
 \frac{1}{2}, 0, \frac{1}{2}, 1 )$ in \eqref{220w2w2w252w02w5w2w0w2w9w2w42w1w2w1w2wPwMw2A}, we obtain 
 the formula of Bauer--Ramanujan type highlighted in \eqref{notMapleBR0}. 
\end{example}

\begin{remark}
 By setting 
\begin{equation}\label{inputBR}
 F(n, k, a) := \frac{ (a)_{k + f} (b)_{k + e} }{ (n)_{k+d} (n)_{k+c}}, 
\end{equation}
 the input function in \eqref{inputBR} satisfies the desired conditions from Section 
 \ref{sectioniterative}, and this can be used to prove the formulas of Bauer--Ramanujan type
 shown in \eqref{MapleBR1} and \eqref{MapleBR2}, 
 by setting 
 $(a, b, c, d, e, f, n) = 
 ( -\frac{1}{4}, -\frac{1}{4}, \frac{1}{4}, \frac{1}{4}, 0, \frac{1}{2}, \frac{3}{4} )$
 and $(a, b, c, d, e, f, n) = 
 ( -\frac{1}{4}, -\frac{1}{4}, \frac{1}{4}, \frac{1}{4}, \frac{1}{2}, 0, \frac{3}{4} )$. 
 However, 
 as noted above, Maple can evaluate the series in 
 \eqref{MapleBR1} and \eqref{MapleBR2}, whereas CAS software such as Maple
 cannot evaluate any of the series given in the Examples 
 of our iteration method. 
\end{remark}

\subsection{Iterative index shifting}
 Observe that the argument of the lower Pochhammer expression $ (n+a)_{k+d}$
 in the input function in \eqref{20250509423P7M7A} is \emph{mixed} 
 in that it is given by a linear combination of the arguments $n$ and $a$ of 
 $\mathfrak{f}(n, a)$ with nonzero coefficients. 
 This leads us to consider a variant of \eqref{20250509423P7M7A}, 
 so that an argument of a $\mathfrak{f}(n, a)$-series 
 obtained through our acceleration method 
 and corresponding to an input function $F(n, k, a)$ would appear in the index of a Pochhammer symbol 
 in $F(n, k, a)$, and hence the term \emph{iterative index shifting}. 
 As we later demonstrate, Theorem \ref{secondmaintheorem}
 can be used to prove new formulas of Ramanujan type
 as in \eqref{furtherclosed1}. 

\begin{theorem}\label{secondmaintheorem}
 The acceleration identity in \eqref{20250507454PM1A} of iteration pattern $(1, 1)$ holds 
 for the input function 
\begin{equation}\label{20q2q5q0q5q0q9q4q2q3qPqq7qMq7A}
 F(n, k, a) := \frac{ (a)_{k + f} (b)_{k + e} }{ (n)_{k + d + a} (n)_{k+c}}. 
\end{equation}
\end{theorem}

\begin{proof}
 Applying Zeilberger's algorithm to \eqref{20q2q5q0q5q0q9q4q2q3qPqq7qMq7A}, 
 we obtain a first-order recurrence of the desired form, with a rational certificate 

 \ 

\noindent $R(n, k, a) = n^2 \big(-a d+a f-a n+b c+b d-b f+b k+2 b n-b-c^2-c d+c e+c f-c k-3 c n+c-d^2+d e+d f-d k-3 d n+d-e f+e k+2 e n-e+f k+2 f n-f-2 k n-3 n^2+2
 n \big)$

 \ 

\noindent and with polynomial coefficients 

 \ 

\noindent $ p_{1}(n, a) = (d-f+n) (a-c+f-n) (-b+c-e+n) (a-b+d-e+n) $ 

 \ 

\noindent and 

 \ 

\noindent $ p_{2}(n, a) = n^2 (b-c-d+e+f-2 n) (b-c-d+e+f-2 n+1)$.

 \ 

\noindent Applying Zeilberger's algorithm to $\mathcal{F}(n, k)$ as specified in 
 \eqref{mathcalFdefine}, this leads to, again being consistent with the notation in Section 
 \ref{sectioniterative}, the secondary rational certificate 

 \ 

\noindent $ \mathcal{R}(n, k, a) = c+k+n-1 $ 

 \ 

\noindent together with the polynomial coefficients 

 \ 

\noindent $ p_{3}(n, a) = a (-a+b-d+e-n) $ 

 \ 

\noindent and 

 \ 

\noindent $ p_{4}(n, a) = a-c+f-n+1$, 

 \ 

\noindent in such a way so that a first-order recurrence of iteration pattern $(1, 1)$ of the desired form holds true, 
 again leading to a series acceleration identity of the form 
\begin{equation*}
 \mathfrak{f}(n, a) = \sum_{j = 0}^{\infty} \left( \prod_{i=0}^{j-1} g_{6}(n+i, a+i) \right) g_{5}(n+j, a+j) 
\end{equation*}
 for $g_{5}$ and $g_{6}$ and $\mathfrak{f}$ as specified in Section \ref{sectioniterative}. 
\end{proof}

\begin{example}
 Setting $(a, b, c, d, e, f, n) = 
 ( -\frac{1}{2}, -\frac{1}{2}, 0, \frac{3}{4}, \frac{3}{4}, \frac{3}{4}, \frac{3}{4} )$ in Theorem \ref{secondmaintheorem}, 
 we obtain 
 \begin{equation*} 
 10 \sqrt{2} = 
 \sum_{j = 0}^{\infty} \left( -\frac{1}{4} \right)^{j} \left[ \begin{matrix} 
 \frac{1}{4}, \frac{3}{8}, \frac{7}{8} 
 \vspace{1mm} \\ 
 1, \frac{9}{8}, \frac{13}{8} 
 \end{matrix} \right]_{j} (160 j^2+132 j+17). 
\end{equation*}
\end{example}

\begin{example}
 Setting $(a, b, c, d, e, f, n) = 
 ( -\frac{1}{6}, -\frac{1}{6}, 1, \frac{1}{2}, -\frac{1}{6}, \frac{5}{6}, \frac{2}{3} )$ in Theorem \ref{secondmaintheorem}, 
 we obtain 
 \begin{equation*} 
 \frac{4 \pi }{9 \sqrt{3}} = 
 \sum_{j = 0}^{\infty} \left( -\frac{1}{4} \right)^{j} \left[ \begin{matrix} 
 \frac{1}{3} 
 \vspace{1mm} \\ 
 \frac{3}{2} 
 \end{matrix} \right]_{j} \frac{(j+1) \left(30 j^2+44 j+13\right)}{(3 j+2) (6 j+7)}. 
\end{equation*}
\end{example}

\begin{example}
 Setting $(a, b, c, d, e, f, n) = 
 ( -\frac{1}{3}, -\frac{1}{3}, \frac{1}{2}, \frac{2}{3}, \frac{1}{2}, \frac{1}{2}, \frac{2}{3} )$ in Theorem \ref{secondmaintheorem}, 
 we obtain 
 \begin{equation*} 
 \frac{\pi }{3} = 
 \sum_{j = 0}^{\infty} \left( -\frac{1}{4} \right)^{j} \left[ \begin{matrix} 
 \frac{5}{6} 
 \vspace{1mm} \\ 
 \frac{3}{2} 
 \end{matrix} \right]_{j} \frac{360 j^3+642 j^2+355 j+61}{(6 j+1) (12 j+5) (12 j+11)}. 
\end{equation*}
\end{example}

\begin{example}
 Setting $(a, b, c, d, e, f, n) = 
 ( -\frac{1}{2}, -\frac{1}{2}, 0, \frac{3}{4}, \frac{1}{4}, \frac{1}{4}, \frac{3}{4} )$ in Theorem \ref{secondmaintheorem}, 
 we obtain 
 \begin{equation*} 
 -\frac{\pi }{2 \sqrt{2}} = 
 \sum_{j = 0}^{\infty} \left( -\frac{1}{4} \right)^{j} \left[ \begin{matrix} 
 \frac{5}{4} 
 \vspace{1mm} \\ 
 \frac{3}{2 } 
 \end{matrix} \right]_{j} \frac{160 j^3+332 j^2+205 j+41}{(4 j-1) (8 j+5) (8 j+9)}. 
\end{equation*}
\end{example}

\begin{example}
 Setting $(a, b, c, d, e, f, n) = 
 ( -\frac{1}{3}, -\frac{1}{3}, \frac{1}{3}, \frac{2}{3}, 0, \frac{2}{3}, \frac{2}{3} )$ in Theorem \ref{secondmaintheorem}, we obtain 
 \begin{equation*} 
 \frac{81 \sqrt{3}}{2 \pi } = 
 \sum_{j = 0}^{\infty} \left( -\frac{1}{4} \right)^{j} \left[ \begin{matrix} 
 \frac{1}{3}, \frac{2}{3}, \frac{2}{3}, \frac{7}{6}, \frac{4}{3} 
 \vspace{1mm} \\ 
 1, 1, \frac{3}{2}, \frac{3}{2}, 2 
 \end{matrix} \right]_{j} (135 j^3+252 j^2+147 j+28). 
\end{equation*}
\end{example}

\begin{example}
 Setting $(a, b, c, d, e, f, n) = 
 ( -\frac{1}{3}, -\frac{1}{3}, \frac{1}{3}, \frac{2}{3}, \frac{2}{3}, 0, \frac{2}{3} )$ in Theorem \ref{secondmaintheorem}, we obtain 
 $$ \frac{81 \sqrt{3}}{2 \pi } = 
 \sum_{j = 0}^{\infty} \left( -\frac{1}{4} \right)^{j} \left[ \begin{matrix} 
 -\frac{1}{3}, \frac{1}{3}, \frac{2}{3}, \frac{5}{6}, \frac{4}{3} 
 \vspace{1mm} \\ 
 1, 1, \frac{3}{2}, \frac{3}{2}, 2 
 \end{matrix} \right]_{j} (135 j^3+261 j^2+138 j+20). $$ 
\end{example}

\begin{example}
 Setting $(a, b, c, d, e, f, n) = 
 ( -\frac{1}{2}, -\frac{1}{2}, \frac{1}{4}, \frac{3}{4}, \frac{1}{4}, \frac{3}{4}, \frac{3}{4} )$ in Theorem \ref{secondmaintheorem}, we obtain 
 \begin{equation*} 
 \frac{256 \sqrt{2}}{\pi } = 
 \sum_{j = 0}^{\infty} \left( -\frac{1}{4} \right)^{j} \left[ \begin{matrix} 
 \frac{1}{4}, \frac{5}{8}, \frac{3}{4}, \frac{9}{8}, \frac{5}{4} 
 \vspace{1mm} \\ 
 1, 1, \frac{3}{2}, \frac{3}{2}, 2 
 \end{matrix} \right]_{j} (640 j^3+1200 j^2+704 j+135). 
\end{equation*}
\end{example}

\begin{example}
 Setting $(a, b, c, d, e, f, n) = 
 ( -\frac{1}{2}, -\frac{1}{2}, \frac{1}{4}, \frac{3}{4}, \frac{3}{4}, \frac{1}{4}, \frac{3}{4} )$ in Theorem \ref{secondmaintheorem}, we obtain 
 \begin{equation*} 
 \frac{256 \sqrt{2}}{\pi } = 
 \sum_{j = 0}^{\infty} \left( -\frac{1}{4} \right)^{j} \left[ \begin{matrix} 
 -\frac{1}{4}, \frac{3}{8}, \frac{3}{4}, \frac{7}{8}, \frac{5}{4} 
 \vspace{1mm} \\ 
 1, 1, \frac{3}{2}, \frac{3}{2}, 2 
 \end{matrix} \right]_{j} (640 j^3+1232 j^2+672 j+105).
\end{equation*}
\end{example}

\begin{example}
 Setting $(a, b, c, d, e, f, n) = 
 ( -\frac{1}{6}, -\frac{1}{6}, \frac{5}{6}, \frac{1}{3}, \frac{2}{3}, \frac{2}{3}, 1 )$ in Theorem \ref{secondmaintheorem}, we obtain 
 the formula of Ramanujan type highlighted in \eqref{furtherclosed1}. 
\end{example}

\begin{example}
 Setting $(a, b, c, d, e, f, n) = 
 ( -\frac{1}{6}, -\frac{1}{6}, 1, 1, \frac{5}{6}, -\frac{1}{6}, \frac{1}{6} )$ in Theorem \ref{secondmaintheorem}, we obtain 
 the formula of Ramanujan type highlighted in \eqref{furthernewRT2}. 
\end{example}

\begin{example}
 Setting $(a, b, c, d, e, f, n) = 
 ( -\frac{1}{6}, -\frac{1}{6}, 0, \frac{5}{6}, \frac{1}{3}, 1, 1 )$ in Theorem \ref{secondmaintheorem}, we obtain 
 the formula of Ramanujan type highlighted in \eqref{2025042529202242P2M2A}. 
\end{example}

\begin{example}
 Setting $(a, b, c, d, e, f, n) = 
 ( -\frac{1}{6}, -\frac{1}{6}, \frac{1}{6}, 1, -\frac{1}{6}, \frac{5}{6}, \frac{5}{6} )$ in Theorem \ref{secondmaintheorem}, we obtain 
 the motivating $\pi$ formula highlighted in \eqref{208828588084828984898P8M1A}. 
\end{example}

\begin{example}
 Setting $(a, b, c, d, e, f, n) = 
 ( -\frac{1}{2}, -\frac{1}{2}, \frac{1}{2}, \frac{1}{4}, \frac{1}{4}, \frac{1}{4}, \frac{1}{4} )$ in Theorem \ref{secondmaintheorem}, we obtain 
 the motivating $\pi$ formula shown in \eqref{20250423motivating452four1}. 
\end{example}

 Theorem \ref{20250517937PM1A} below also relies on iterative index shifted, 
 and it can be applied, as below, to prove the formula of Ramanujan type highlighted in \eqref{furthernewRT4}.

\begin{theorem}\label{20250517937PM1A}
 The acceleration identity in \eqref{20250507454PM1A} of iteration pattern $(1, 1)$ holds 
 for the input function 
 $$ F(n, k, a) := \frac{ (d)_{k + a} (b)_{k+e} }{ (n+a)_{k+f} (n)_{k+c}}. $$
\end{theorem}

\begin{proof}
 By analogy with the proofs of Theorems \ref{firstmaintheorem} and 
 \ref{secondmaintheorem}, this can be shown to hold using the rational certificate 

 \ 

\noindent $R(n, k, a) = -n (a+n) \big(-a d+a f+a n-b c+b d-b f-b k-2 b n+b+c^2-c d-c e+c f+c k+3 c n-c + 
 d e-d f-d k-2 d n+d-e f-e k-2 e n+e+f^2+f k+3 f n-f+2 k n+3 n^2-2 n \big)$ 

 \ 

\noindent together with the secondary rational certificate 

 \ 

\noindent $\mathcal{R}(n, k, a) = (a+n) (c+k+n-1)$

 \ 

\noindent according to the notation in Section \ref{sectioniterative}, leaving the details to the reader. 
 \end{proof}

\begin{example}
 Setting $(a, b, c, d, e, f, n) = 
 ( -\frac{1}{4}, -\frac{1}{4}, -\frac{1}{4}, -\frac{1}{2}, -\frac{1}{2}, -\frac{1}{4}, \frac{1}{2} )$ in Theorem \ref{20250517937PM1A}, we obtain 
\begin{equation*} 
 -\frac{3 \pi }{8 \sqrt{2}} = 
 \sum_{j = 0}^{\infty} \left( -\frac{1}{4} \right)^{j} \left[ \begin{matrix} 
 \frac{3}{4} 
 \vspace{1mm} \\ 
 \frac{1}{2}
 \end{matrix} \right]_{j} \frac{160 j^3+148 j^2+21 j+9}{(4 j-3) (8 j+3) (8 j+7)}. 
\end{equation*}
\end{example}

\begin{example}
 Setting $(a, b, c, d, e, f, n) = 
 ( \frac{1}{2}, -\frac{1}{2}, \frac{1}{2}, -\frac{1}{3}, \frac{1}{3}, 0, \frac{1}{2} )$ in Theorem \ref{20250517937PM1A}, 
 we obtain 
 \begin{equation*} 
 \frac{1296}{\pi } = 
 \sum_{j = 0}^{\infty} \left( -\frac{1}{4} \right)^{j} \left[ \begin{matrix} 
 \frac{1}{6}, \frac{7}{12}, \frac{5}{6}, \frac{13}{12}, \frac{7}{6} 
 \vspace{1mm} \\ 
 1, 1, \frac{3}{2}, \frac{3}{2}, 2 
 \end{matrix} \right]_{j} (2160 j^3+4068 j^2+2388 j+455). 
\end{equation*}
\end{example}

\begin{example}
 Setting $(a, b, c, d, e, f, n) = 
 ( \frac{1}{2}, -\frac{1}{2}, \frac{1}{2}, \frac{1}{3}, -\frac{1}{3}, 0, \frac{1}{2} )$ in Theorem \ref{20250517937PM1A}, 
 we obtain 
 \begin{equation*} 
 \frac{1296}{5 \pi } = 
 \sum_{j = 0}^{\infty} \left( -\frac{1}{4} \right)^{j} \left[ \begin{matrix} 
 \frac{1}{6}, \frac{5}{6}, \frac{11}{12}, \frac{17}{12}, \frac{11}{6} 
 \vspace{1mm} \\ 
 1, 1, \frac{3}{2}, \frac{3}{2}, 2 
 \end{matrix} \right]_{j} (2160 j^3+3924 j^2+1956 j+187). 
\end{equation*}
\end{example}

\begin{example}
 Setting $(a, b, c, d, e, f, n) = 
 ( -\frac{1}{2}, \frac{7}{8}, \frac{7}{8}, \frac{7}{8}, \frac{1}{4}, \frac{3}{4}, \frac{3}{4} )$ 
 in Theorem \ref{20250517937PM1A}, 
 we obtain the formula of Ramanujan type highlighted in \eqref{furthernewRT4}. 
\end{example}

  By considering the use of iterated index shifting by  reversing the sign of the $a$-term involved in the Pochhammer indices above,  
  this has led us to experimentally discover  how the input function in \eqref{27027507517437332AM2A} yields, via our iterative method,  
  accelerated series with a positive convergence rate of $\frac{1}{4}$, 
 in contrast to the convergence rate obtained from 
 both Theorem \ref{firstmaintheorem} and Theorem \ref{secondmaintheorem}. 

\begin{theorem}\label{theorem2nnega}
 The acceleration identity in \eqref{20250507454PM1A} of iteration pattern $(1, 1)$ holds 
 for the input function 
\begin{equation}\label{27027507517437332AM2A}
 F(n, k, a) := \frac{ (b)_{k + f} (c)_{k +e} }{ (2n)_{k-a} (n)_{k+d}}. 
\end{equation}
\end{theorem}

\begin{proof}
   This can be shown to hold by   analogy with the proofs of Theorems \ref{firstmaintheorem} and  
   \ref{secondmaintheorem},  using the   iteration pattern $(1, 1)$, and leaving the details to the reader. 
\end{proof}

\begin{example}
 Setting $(a, b, c, d, e, f, n) = 
 ( -\frac{1}{2}, -\frac{1}{2}, -\frac{1}{2}, -\frac{1}{6}, -\frac{1}{6}, \frac{2}{3}, \frac{1}{6} )$ in Theorem \ref{theorem2nnega}, 
 we obtain 
 \begin{equation*} 
 \frac{8 \sqrt[3]{2}}{27} = 
 \sum_{j = 0}^{\infty} \left( \frac{1}{4} \right)^{j} \left[ \begin{matrix} 
 \frac{2}{3}, \frac{3}{2} 
 \vspace{1mm} \\ 
 1, \frac{7}{6}
 \end{matrix} \right]_{j}   j.  
 \end{equation*}
\end{example}

\begin{example}
 Setting $(a, b, c, d, e, f, n) = 
 ( -\frac{1}{2}, -\frac{1}{2}, -\frac{1}{2}, \frac{1}{3}, -\frac{1}{6}, \frac{2}{3}, \frac{2}{3} )$ in Theorem \ref{theorem2nnega}, 
 we obtain 
 \begin{equation*} 
 \frac{112 \sqrt[3]{2}}{81} = 
 \sum_{j = 0}^{\infty} \left( \frac{1}{4} \right)^{j} \left[ \begin{matrix} 
 \frac{5}{3}, \frac{5}{2} 
 \vspace{1mm} \\ 
 1, \frac{13}{6} 
 \end{matrix} \right]_{j}. 
 \end{equation*}
\end{example}

\begin{example}
 Setting $(a, b, c, d, e, f, n) = 
 ( -\frac{1}{2}, -\frac{1}{2}, -\frac{1}{2}, -\frac{1}{6}, 0, \frac{2}{3}, \frac{1}{6} )$ in Theorem \ref{theorem2nnega}, we obtain 
 \begin{equation*} 
 \frac{7 \sqrt{3}}{8} = 
 \sum_{j = 0}^{\infty} \left( \frac{1}{4} \right)^{j} \left[ \begin{matrix} 
 \frac{1}{2}, \frac{2}{3}, \frac{4}{3} 
 \vspace{1mm} \\ 
 1, \frac{13}{12}, \frac{19}{12} 
 \end{matrix} \right]_{j} 
 j (9 j+5). 
 \end{equation*}
\end{example}

\begin{example}
 Setting $(a, b, c, d, e, f, n) = 
 ( -\frac{1}{3}, -\frac{1}{3}, -\frac{1}{3}, 0, -\frac{1}{3}, -\frac{1}{3}, \frac{1}{3} )$ in Theorem \ref{theorem2nnega}, we obtain 
 \begin{equation*} 
 \frac{160 \pi }{9 \sqrt{3}} = 
 \sum_{j = 0}^{\infty} \left( \frac{1}{4} \right)^{j} \left[ \begin{matrix} 
 1, \frac{5}{3}, \frac{5}{3} 
 \vspace{1mm} \\ 
 \frac{1}{3}, \frac{11}{6}, \frac{7}{3} 
 \end{matrix} \right]_{j} 
 (9 j^2+14 j+3). 
 \end{equation*}
\end{example}

\begin{example}
 Setting $(a, b, c, d, e, f, n) = 
 ( -\frac{1}{3}, -\frac{1}{3}, -\frac{1}{3}, \frac{1}{3}, 0, 0, \frac{1}{3} )$ in Theorem \ref{theorem2nnega}, we obtain 
 \begin{equation*} 
 \frac{56 \pi }{9 \sqrt{3}} = 
 \sum_{j = 0}^{\infty} \left( \frac{1}{4} \right)^{j} \left[ \begin{matrix} 
 1, \frac{4}{3}, \frac{4}{3} 
 \vspace{1mm} \\ 
 \frac{2}{3}, \frac{5}{3}, \frac{13}{6}
 \end{matrix} \right]_{j} 
 (9 j^2+13 j+3). 
 \end{equation*}
\end{example}

\begin{example}
 Setting $(a, b, c, d, e, f, n) = 
 ( -\frac{1}{3}, -\frac{1}{3}, -\frac{1}{3}, \frac{1}{3}, 1, 1, \frac{2}{3} )$ in Theorem \ref{theorem2nnega}, we obtain 
 \begin{equation*} 
 \frac{16 \pi }{9 \sqrt{3}} = 
 \sum_{j = 0}^{\infty} \left( \frac{1}{4} \right)^{j} \left[ \begin{matrix} 
 \frac{1}{3}, \frac{1}{3}, 1 
 \vspace{1mm} \\ 
 \frac{7}{6}, \frac{5}{3}, \frac{5}{3} 
 \end{matrix} \right]_{j} 
 (9 j^2+10 j+3). 
 \end{equation*}
\end{example}

\begin{example}
 Setting $(a, b, c, d, e, f, n) = 
 ( -\frac{1}{3}, -\frac{1}{3}, -\frac{1}{3}, \frac{1}{3}, 0, 1, \frac{2}{3} )$ in Theorem \ref{theorem2nnega}, we obtain 
 \begin{equation*} 
 \frac{112 \pi }{27 \sqrt{3}} = 
 \sum_{j = 0}^{\infty} \left( \frac{1}{4} \right)^{j} \left[ \begin{matrix} 
 \frac{1}{3}, \frac{4}{3}, 2 
 \vspace{1mm} \\ 
 \frac{5}{3}, \frac{5}{3}, \frac{13}{6} 
 \end{matrix} \right]_{j} 
 (9 j^2+16 j+6). 
 \end{equation*}
\end{example}

\begin{example}
 Setting $(a, b, c, d, e, f, n) = 
 ( -\frac{1}{3}, -\frac{1}{3}, -\frac{1}{3}, 1, -\frac{1}{3}, \frac{2}{3}, \frac{1}{3} )$ in Theorem \ref{theorem2nnega}, we obtain 
 \begin{equation*} 
 \frac{160 \pi }{27 \sqrt{3}} = 
 \sum_{j = 0}^{\infty} \left( \frac{1}{4} \right)^{j} \left[ \begin{matrix} 
 \frac{2}{3}, \frac{5}{3}, 2 
 \vspace{1mm} \\ 
 \frac{4}{3}, \frac{11}{6}, \frac{7}{3} 
 \end{matrix} \right]_{j} 
 (9 j^2+17 j+6). 
 \end{equation*}
\end{example}

\begin{example}
 Setting $(a, b, c, d, e, f, n) = 
 ( -\frac{1}{3}, -\frac{1}{3}, -\frac{1}{3}, 1, \frac{2}{3}, \frac{2}{3}, \frac{1}{3} )$ in Theorem \ref{theorem2nnega}, we obtain 
 \begin{equation*} 
 \frac{20 \pi }{9 \sqrt{3}} = 
 \sum_{j = 0}^{\infty} \left( \frac{1}{4} \right)^{j} \left[ \begin{matrix} 
 \frac{2}{3}, \frac{2}{3}, 1 
 \vspace{1mm} \\ 
 \frac{4}{3}, \frac{4}{3}, \frac{11}{6} 
 \end{matrix} \right]_{j} 
 (9 j^2+11 j+3). 
 \end{equation*}
\end{example}

\begin{example}
 Setting $(a, b, c, d, e, f, n) = 
 ( -\frac{1}{2}, -\frac{1}{2}, -\frac{1}{6}, \frac{1}{3}, 1, \frac{1}{3}, \frac{2}{3} )$ in Theorem \ref{theorem2nnega}, we obtain 
 \begin{equation*} 
 \frac{455 \pi }{108} = 
 \sum_{j = 0}^{\infty} \left( \frac{1}{4} \right)^{j} \left[ \begin{matrix} 
 \frac{1}{6}, \frac{7}{6}, 2 
 \vspace{1mm} \\ 
 \frac{19}{12}, \frac{11}{6}, \frac{25}{12} 
 \end{matrix} \right]_{j} 
 (18 j^2+31 j+12). 
 \end{equation*}
\end{example}

\begin{example}
 Setting $(a, b, c, d, e, f, n) = 
 ( -\frac{1}{3}, -\frac{1}{3}, -\frac{1}{3}, \frac{2}{3}, 0, \frac{2}{3}, \frac{1}{3} )$ in Theorem \ref{theorem2nnega}, we obtain 
 \begin{equation*} 
 \frac{27 \sqrt{3}}{\pi } = 
 \sum_{j = 0}^{\infty} \left( \frac{1}{4} \right)^{j} \left[ \begin{matrix} 
 \frac{2}{3}, \frac{2}{3}, \frac{4}{3}, \frac{4}{3} 
 \vspace{1mm} \\ 
 1, 1, \frac{3}{2}, 2 
 \end{matrix} \right]_{j} 
 (27 j^2+36 j+8). 
 \end{equation*}
\end{example}

\begin{example}
 Setting $(a, b, c, d, e, f, n) = 
 ( -\frac{1}{3}, -\frac{1}{3}, -\frac{1}{3}, \frac{2}{3}, 1, -\frac{1}{3}, \frac{1}{3} )$ in Theorem \ref{theorem2nnega}, we obtain 
 \begin{equation*} 
 \frac{27 \sqrt{3}}{2 \pi } = 
 \sum_{j = 0}^{\infty} \left( \frac{1}{4} \right)^{j} \left[ \begin{matrix} 
 \frac{1}{3}, \frac{1}{3}, \frac{5}{3}, \frac{5}{3} 
 \vspace{1mm} \\ 
 1, 1, \frac{3}{2}, 2 
 \end{matrix} \right]_{j} 
 (27 j^2+36 j+5). 
 \end{equation*}
\end{example}

\begin{example}
 Setting $(a, b, c, d, e, f, n) = 
 ( -\frac{1}{2}, -\frac{1}{2}, -\frac{1}{2}, \frac{1}{2}, \frac{1}{4}, \frac{1}{4}, \frac{1}{4} )$ in Theorem \ref{theorem2nnega}, we obtain 
 \begin{equation*} 
 \frac{45 \pi }{8 \sqrt{2}} = 
 \sum_{j = 0}^{\infty} \left( \frac{1}{4} \right)^{j} \left[ \begin{matrix} 
 1, \frac{5}{4}, \frac{5}{4} 
 \vspace{1mm} \\ 
 \frac{3}{4}, \frac{13}{8}, \frac{17}{8} 
 \end{matrix} \right]_{j} 
 (12 j^2+17 j+4). 
 \end{equation*}
\end{example}

\begin{example}
 Setting $(a, b, c, d, e, f, n) = 
 ( -\frac{1}{2}, -\frac{1}{2}, -\frac{1}{2}, \frac{1}{4}, \frac{3}{4}, \frac{1}{4}, \frac{3}{4} )$ in Theorem \ref{theorem2nnega}, we obtain 
 \begin{equation*} 
 \frac{1024 \sqrt{2}}{5 \pi } = 
 \sum_{j = 0}^{\infty} \left( \frac{1}{4} \right)^{j} \left[ \begin{matrix} 
 \frac{3}{4}, \frac{5}{4}, \frac{7}{4}, \frac{9}{4} 
 \vspace{1mm} \\ 
 1, 2, 2, \frac{5}{2} 
 \end{matrix} \right]_{j} 
 (48 j^2+112 j+63). 
 \end{equation*}
\end{example}

\begin{example}
 Setting $(a, b, c, d, e, f, n) = 
 ( -\frac{1}{2}, -\frac{1}{2}, \frac{1}{4}, \frac{1}{4}, \frac{1}{2}, \frac{3}{4}, \frac{3}{4} )$ in Theorem \ref{theorem2nnega}, we obtain 
 \begin{equation*} 
 \frac{256 \sqrt{2}}{3 \pi } = 
 \sum_{j = 0}^{\infty} \left( \frac{1}{4} \right)^{j} \left[ \begin{matrix} 
 \frac{1}{4}, \frac{3}{4}, \frac{5}{4}, \frac{7}{4} 
 \vspace{1mm} \\ 
 1, \frac{3}{2}, 2, 2 
 \end{matrix} \right]_{j} 
 (48j^2 + 80 j + 35). 
 \end{equation*}
\end{example}

\begin{example}
 Setting $(a, b, c, d, e, f, n) = 
 ( -\frac{1}{2}, -\frac{1}{2}, \frac{1}{4}, \frac{1}{2}, \frac{3}{4}, 1, \frac{3}{4} )$ in Theorem \ref{theorem2nnega}, we obtain 
 \begin{equation*} 
 \frac{21 \pi }{2} = 
 \sum_{j = 0}^{\infty} \left( \frac{1}{4} \right)^{j} \left[ \begin{matrix} 
 \frac{3}{4}, \frac{3}{2} 
 \vspace{1mm} \\ 
 \frac{11}{8}, \frac{15}{8} 
 \end{matrix} \right]_{j} 
 \frac{48 j^2+76 j+31}{(j+1) (4 j+1)}. 
 \end{equation*}
\end{example}

\begin{example}
 Setting $(a, b, c, d, e, f, n) = 
 ( -\frac{1}{2}, -\frac{1}{2}, \frac{1}{4}, \frac{1}{2}, -\frac{1}{2}, \frac{1}{4}, \frac{1}{4} )$ in Theorem \ref{theorem2nnega}, we obtain 
 \begin{equation*} 
 \frac{45 \pi }{8 \sqrt{2}} = 
 \sum_{j = 0}^{\infty} \left( \frac{1}{4} \right)^{j} \left[ \begin{matrix} 
 1, \frac{5}{4}, \frac{5}{4} 
 \vspace{1mm} \\ 
 \frac{3}{4}, \frac{13}{8}, \frac{17}{8} 
 \end{matrix} \right]_{j} 
 (12 j^2+17 j+4). 
 \end{equation*}
\end{example}

\begin{example}
 Setting $(a, b, c, d, e, f, n) = 
 ( -\frac{1}{2}, -\frac{1}{2}, -\frac{1}{2}, \frac{5}{6}, \frac{1}{3}, \frac{1}{3}, \frac{1}{6} )$ in Theorem \ref{theorem2nnega}, we obtain 
 \begin{equation*} 
 \frac{91 \pi }{18} = 
 \sum_{j = 0}^{\infty} \left( \frac{1}{4} \right)^{j} \left[ \begin{matrix} 
 1, \frac{7}{6}, \frac{7}{6} 
 \vspace{1mm} \\ 
 \frac{5}{6}, \frac{19}{12}, \frac{25}{12} 
 \end{matrix} \right]_{j} 
 (18 j^2+25 j+6). 
 \end{equation*}
\end{example}

\begin{example}
 Setting $(a, b, c, d, e, f, n) = 
 ( -\frac{1}{2}, -\frac{1}{2}, \frac{1}{4}, \frac{3}{4}, 0, \frac{1}{4}, \frac{1}{4} )$ in Theorem \ref{theorem2nnega}, we obtain 
 \begin{equation*} 
 \frac{64 \sqrt{2}}{\pi } = 
 \sum_{j = 0}^{\infty} \left( \frac{1}{4} \right)^{j} \left[ \begin{matrix} 
 \frac{3}{4}, \frac{3}{4}, \frac{5}{4}, \frac{5}{4} 
 \vspace{1mm} \\ 
 1, 1, \frac{3}{2}, 2 
 \end{matrix} \right]_{j} 
 (48 j^2+64 j+15). 
 \end{equation*}
\end{example}

\begin{example}
 Setting $(a, b, c, d, e, f, n) = 
 ( -\frac{1}{2}, -\frac{1}{2}, -\frac{1}{2}, 1, \frac{1}{3}, \frac{2}{3}, \frac{5}{6} )$ in Theorem \ref{theorem2nnega}, we obtain 
 the formula of Ramanujan type highlighted in \eqref{pquarterRT1}.
\end{example}

\begin{example}
 Setting $(a, b, c, d, e, f, n) = 
 ( -\frac{1}{2}, -\frac{1}{2}, \frac{1}{4}, 1, \frac{1}{2}, \frac{1}{4}, \frac{1}{4} )$ in Theorem \ref{theorem2nnega}, we obtain 
 the formula of Ramanujan type highlighted in \eqref{pquarterRT2}. 
\end{example}

\begin{example}
 Setting $(a, b, c, d, e, f, n) = 
 ( -\frac{1}{2}, -\frac{1}{2}, -\frac{1}{2}, \frac{1}{6}, 0, \frac{1}{3}, \frac{5}{6} )$ in Theorem \ref{theorem2nnega}, we obtain 
 the formula of Ramanujan type highlighted in \eqref{202505121101PM1A}. 
\end{example}

\begin{example}
 Setting $(a, b, c, d, e, f, n) = 
 ( -\frac{1}{2}, -\frac{1}{2}, -\frac{1}{2}, \frac{1}{6}, \frac{1}{3}, 1, \frac{5}{6} )$ in Theorem \ref{theorem2nnega}, we obtain 
 the formula of Ramanujan type highlighted in \eqref{202505121118PM1A}. 
\end{example}
 
\begin{example}
 Setting $(a, b, c, d, e, f, n) = 
 ( -\frac{1}{2}, -\frac{1}{2}, -\frac{1}{2}, \frac{1}{6}, 0, \frac{1}{3}, \frac{5}{6} )$ in Theorem \ref{theorem2nnega}, we obtain 
 the formula of Ramanujan type highlighted in \eqref{nRTsqrt39p16}.
\end{example}

\begin{example}
 Setting $(a, b, c, d, e, f, n) = 
 ( -\frac{1}{2}, -\frac{1}{2}, -\frac{1}{2}, \frac{1}{3}, 0, \frac{2}{3}, \frac{2}{3} )$ in Theorem \ref{theorem2nnega}, we obtain 
 the formula of Ramanujan type highlighted in \eqref{202505121131PM1A}.
\end{example}
 
 \begin{example}
 Setting $(a, b, c, d, e, f, n) = 
 ( -\frac{1}{3}, -\frac{1}{3}, -\frac{1}{3}, 1, \frac{1}{2}, 1, \frac{1}{3} )$ in Theorem \ref{theorem2nnega}, we obtain 
 the formula of Ramanujan type highlighted in \eqref{20250513745PM1AAM1A}. 
\end{example}

\begin{example}
 Setting $(a, b, c, d, e, f, n) = 
 ( -\frac{1}{3}, -\frac{1}{3}, \frac{1}{3}, \frac{1}{3}, 0, \frac{1}{2}, \frac{2}{3} )$ in Theorem \ref{theorem2nnega}, we obtain 
 the formula of Ramanujan type highlighted in \eqref{20250513856}.
\end{example}

\begin{example}
 Setting $(a, b, c, d, e, f, n) = 
 ( -\frac{1}{12}, -\frac{1}{12}, -\frac{1}{12}, \frac{1}{12}, \frac{1}{6}, \frac{1}{2}, \frac{11}{12} )$ in Theorem \ref{theorem2nnega}, we obtain 
 the formula of Ramanujan type highlighted in \eqref{20250513323AM1A}.
\end{example}

\begin{example}
 Setting $(a, b, c, d, e, f, n) = 
 ( -\frac{1}{2}, -\frac{1}{2}, -\frac{1}{2}, 1, 1, 1, \frac{1}{2} )$ in Theorem \ref{theorem2nnega}, we 
 recover the formula of Ramanujan type due to Fabry and Guillera
 shown in \eqref{FabryGuillera}. 
\end{example}

\begin{example}
 Setting $(a, b, c, d, e, f, n) = 
 ( -\frac{1}{2}, -\frac{1}{2}, -\frac{1}{2}, \frac{1}{2}, 0, 1, 1 )$ in Theorem \ref{theorem2nnega}, we 
 the formula of Ramanujan type highlighted in \eqref{GuilleraGosper}. 
\end{example}

\begin{example}
 Setting $(a, b, c, d, e, f, n) = 
 ( -\frac{1}{2}, -\frac{1}{2}, \frac{1}{4}, \frac{1}{4}, \frac{1}{4}, 0, \frac{3}{4} )$ in Theorem \ref{theorem2nnega}, we obtain 
 the formula of Ramanujan type highlighted in \eqref{7202757075414347465837PM2A}. 
\end{example}

\begin{example}
 Setting $(a, b, c, d, e, f, n) = 
 ( -\frac{1}{2}, -\frac{1}{2}, -\frac{1}{2}, \frac{3}{4}, 0, 0, \frac{1}{4} )$ in Theorem \ref{theorem2nnega}, we obtain 
 the formula of Ramanujan type highlighted in \eqref{quasinewRT2p1}. 
\end{example}

\begin{example}
 Setting $(a, b, c, d, e, f, n) = 
 ( -\frac{1}{2}, -\frac{1}{2}, -\frac{1}{2}, \frac{5}{6}, \frac{1}{3}, \frac{1}{3}, \frac{1}{6} )$ in Theorem \ref{theorem2nnega}, we obtain 
 the motivating $\pi$ formula highlighted in \eqref{20727570757173310AM1A}. 
\end{example}

 \subsection{Unary iterations}
 The repeated application of the recursion in \eqref{iteration1c0} 
 provides a ``unary'' version of our extended and iterative method. 
 According to the notation in \eqref{iteration1c0} 
 for the bivariate function $\mathfrak{f}(n, a)$, the recurrence in 
 \eqref{iteration1c0} is of iteration pattern $(1, 0)$. 
 By writing $\mathfrak{f}(n, a) = \mathfrak{f}(n)$ in this case, 
 we obtain a recurrence of iteration pattern $(1)$. 
 For a bivariate hypergeometric function $F(n, k)$ that satisfies the desired conditions 
 from our past work concerning Wilf's method
 \cite{Campbellunpublished,CampbellLevrie2024free} 
 algorithm and that does not involve a parameter $\alpha$, by writing $F(n, k, \alpha) := F(n, k)$ 
 for a free parameter $\alpha$, the degenerate recursion $F(n, k, \alpha) = F(n, k, \alpha + 1)$ 
 is such that a modified version of the iterative 
 approach in Section \ref{sectioniterative} reduces to the unary case, and hence how our iterative method extends 
 the past research works on the unary case \cite{Campbellunpublished,CampbellLevrie2024free,Wilf1999}. 
 So, it would be consistent, in regard to the application of our iterative method, to further consider the unary case. 
 Moreover, since formulas of Ramanujan type, as defined in 
 Section \ref{sectiontype}, are a main object of study in this paper, 
 this motivates our derivation of new formulas of this type 
 through the unary case. 

 By inputting 
\begin{equation}\label{inputfirstunary}
 F(n, k) := \frac{ (n-a)_{k+f} (b)_{k+e} }{ (n+a)_{k+d} (2n)_{k+c} } 
\end{equation}
 using Zeilberger's algorithm, we obtain a first-order recurrence of the desired form, 
 with a rational certificate 

 \ 

\noindent $R(n, k) = \frac{1}{c+k+2 n} 2 n (2 n+1) (a+n) \big(-a^2 b+3 a^2 c-a^2 e+2 a^2 k+6 a^2 n+a b^2-4 a b c-a b d+2 a b e + 
 a b f-2 a b k-8 a b n+3 a c^2+4 a c d-4 a c e-2 a c f+4 a c k+14 a c n-2 a c-a d e+3 a d k+8 a d n+a e^2+a e f-2 a e k-8 a e n - 
 a f k-4 a f n+2 a k^2+10 a k n-2 a k+16 a n^2-4 a n+b^2 c-b^2 f+b^2 n-2 b c^2-2 b c d+2 b c e+2 b c f-2 b c k-8 b c n+b c+b d f - 
 b d k-3 b d n-2 b e f+2 b e n+b f k+5 b f n-b k^2-4 b k n+b k-7 b n^2+2 b n+c^3+2 c^2 d-2 c^2 e-c^2 f+2 c^2 k+7 c^2 n-c^2 +
 c d^2-2 c d e-2 c d f+2 c d k+8 c d n-c d+c e^2+2 c e f-2 c e k-8 c e n+c e-2 c f k-6 c f n+c f+c k^2+8 c k n-c k+15 c n^2-4 c n + 
 d^2 k+2 d^2 n+d e f-d e k-3 d e n-d f k-4 d f n+d k^2+5 d k n-d k+8 d n^2-2 d n-e^2 f+e^2 n+e f k+5 e f n-e k^2-4 e k n + 
 e k-7 e n^2+2 e n-f k^2-5 f k n+f k-8 f n^2+2 f n+2 k^2 n+8 k n^2-2 k n+10 n^3-4 n^2 \big) $ 

 \ 

\noindent and with polynomial coefficients $$ p_{1}(n) = (-1 + b - c + e - 2 n) (b - c + e - 2 n) (a - n) (a - b + d - e + 
 n) (a + c - f + n)$$ and $$p_{2}(n) = -2 n (2 n+1) (a+n) (2 a-b+c+d-e-f+2 n-1) (2 a-b+c+d-e-f+2 n)$$ 
 yielding a first-order recurrence 
 $p_{1}(n) F(n+1, k) + p_{2}(n) F(n, k) = G(n, k+1) - G(n, k)$ of the desired form. 
 Through the repeated application of an equivalent 
 ``unary'' version of the recurrence in \eqref{iteration1c0}, 
 based on the input function in \eqref{inputfirstunary}, 
 this can be used to prove the following results, leaving the details concerning 
 the unary case under consideration to the reader. 

\begin{example}
 Setting $(a, b, c, d, e, f, n) = 
 ( -\frac{1}{2}, -\frac{1}{2}, 1, \frac{3}{4}, \frac{1}{4}, 0, \frac{3}{4} )$ after applying the unary acceleration
 method to \eqref{inputfirstunary}, we obtain 
 the formula of Ramanujan type highlighted in \eqref{newRT5}. 
\end{example}

\begin{example}
 Setting $(a, b, c, d, e, f, n) = 
 ( -\frac{1}{2}, -\frac{1}{2}, 1, \frac{1}{2}, 0, -\frac{1}{3}, \frac{1}{3} )$ after applying the unary acceleration
 method to \eqref{inputfirstunary}, we obtain 
 the formula of Ramanujan type shown in \eqref{newRT9}. 
\end{example}

\begin{example}
 Setting $(a, b, c, d, e, f, n) = 
 ( -\frac{1}{2}, -\frac{1}{2}, 0, 1, 1, -\frac{1}{2}, 1 )$ after applying the unary acceleration method to \eqref{inputfirstunary}, 
 we obtain 
 the formula of Ramanujan type highlighted in \eqref{newRT10}. 
\end{example}

 Using our unary acceleration method with 
\begin{equation}\label{secondunary}
 F(n, k) := \frac{ (a-n)_{k + f} (b)_{k+e} }{ (n+a)_{k+d} (n)_{k+c} } 
\end{equation}
 as the input for Zeilberger's algorithm, 
 this has led us to discover further formulas of Ramanujan type, 
 for the specified parameter combinations
  below,  
  again leaving the details to the reader. 

\begin{example}
 Setting $(a, b, c, d, e, f, n) = 
 ( -\frac{1}{2}, -\frac{1}{2}, \frac{3}{2}, \frac{3}{2}, \frac{1}{4}, \frac{3}{2}, \frac{3}{4} )$ 
 after applying the unary acceleration method to \eqref{secondunary}, 
 we obtain 
 \begin{equation*} 
 15 \pi = 
 \sum_{j = 0}^{\infty} \left( \frac{16}{27} \right)^{j} \left[ \begin{matrix} 
 1, \frac{3}{2}, \frac{5}{2} 
 \vspace{1mm} \\ 
 \frac{7}{4}, \frac{7}{3}, \frac{8}{3} 
 \end{matrix} \right]_{j} 
 (11 j^2+29 j+16). 
 \end{equation*}
\end{example}

\begin{example}
 Setting $(a, b, c, d, e, f, n) = 
 ( \frac{1}{2}, \frac{1}{2}, \frac{1}{2}, \frac{1}{2}, \frac{1}{2}, 1, 1 )$ after applying
 the unary acceleration method to \eqref{secondunary}, we obtain 
 \begin{equation*} 
 12 \log (2) = 
 \sum_{j = 0}^{\infty} \left( \frac{16}{27} \right)^{j} \left[ \begin{matrix} 
 \frac{3}{4}, \frac{5}{4} 
 \vspace{1mm} \\ 
 \frac{4}{3}, \frac{5}{3} 
 \end{matrix} \right]_{j} 
 \frac{11 j^2+16 j+6}{(j+1) (2 j+1)}. 
 \end{equation*}
\end{example}

\begin{example}
 Setting $(a, b, c, d, e, f, n) = 
 ( \frac{1}{2}, \frac{1}{2}, 0, \frac{1}{2}, 0, 0, 1 )$ after applying the unary acceleration method to \eqref{secondunary}, 
 we obtain 
 \begin{equation*} 
 \frac{512}{\pi } = 
 \sum_{j = 0}^{\infty} \left( \frac{16}{27} \right)^{j} \left[ \begin{matrix} 
 \frac{1}{2}, \frac{3}{4}, \frac{5}{4}, \frac{5}{4}, \frac{7}{4} 
 \vspace{1mm} \\ 
 1, \frac{5}{3}, 2, 2, \frac{7}{3} 
 \end{matrix} \right]_{j} 
 (88 j^3+284 j^2+300 j+105). 
 \end{equation*}
\end{example}
 
\begin{example}
 Setting $(a, b, c, d, e, f, n) = 
 ( -\frac{1}{2}, -\frac{1}{2}, \frac{3}{2}, \frac{3}{2}, \frac{3}{4}, \frac{3}{2}, \frac{1}{4} )$ after applying the unary acceleration 
 method to \eqref{secondunary}, we obtain 
 the formula of Ramanujan type highlighted in \eqref{newRamanujantype1}. 
\end{example}

\begin{example}
 Setting $(a, b, c, d, e, f, n) = 
 ( \frac{1}{2}, \frac{1}{2}, \frac{1}{2}, -\frac{1}{2}, 0, 1, 1 )$ after applying the unary acceleration method 
 to \eqref{secondunary}, we obtain 
 the formula of Ramanujan type shown in \eqref{recoveranother}. 
\end{example}

\begin{example}
 Setting $(a, b, c, d, e, f, n) = 
 ( \frac{1}{2}, \frac{1}{2}, \frac{1}{2}, 0, 0, 1, 1 )$ after applying the unary acceleration method to \eqref{secondunary}, 
 we obtain the formula of Ramanujan type highlighted in \eqref{202540429935PM1A}. 
\end{example}

\begin{example}
 Setting $(a, b, c, d, e, f, n) = 
 ( -\frac{1}{2}, -\frac{1}{2}, \frac{1}{2}, \frac{1}{2}, 1, 0, 1 )$ after applying the unary acceleration method
 to \eqref{secondunary}, we obtain 
 the motivating $\pi$ formula highlighted in \eqref{2025043072226PM2A}. 
\end{example}

\subsection{A ternary iteration}
 By modifying our derivation of the recurrence in \eqref{202505067425PM2A} 
 of iteration pattern $(1, 1)$, we make use of an input function
 $F(n, k, a, b)$ in such a way so that, by writing 
 $\mathfrak{f}(n, a, b) := \sum_{k=0}^{\infty} F(n, k, a, b)$, we would obtain a recurrence of the form
\begin{equation}\label{202505191233PM1A}
 \mathfrak{f}(n, a, b) = g_{7}(n, a, b) 
 + g_{8}(n, a, b) \, \mathfrak{f}(n+1, a+1, b+1) 
\end{equation}
 of iteration pattern $(1, 1, 1)$, providing a ternary version of 
 of accelerations based on the interation pattern $(1, 1)$. 
 Leaving the details to the reader, through the repeated application of 
 \eqref{202505191233PM1A} based on 
\begin{equation}\label{inputternary}
 F(n, k, a, b) := \frac{ (a)_{k+f} (b)_{k+e} }{ (n+a)_{k+d} (n+b)_{k+c} }, 
\end{equation}
 this leads to accelerated series of convergence rate $\frac{1}{64}$, as below. 

\begin{example}
 Setting $(a, b, c, d, e, f, n) = 
 ( \frac{1}{2}, \frac{1}{2}, \frac{1}{2}, \frac{1}{2}, \frac{1}{2}, \frac{1}{2}, 1 )$
 after applying our acceleration method to \eqref{inputternary}, 
 we recover the formula of Ramanujan type highlighted in \eqref{202505191141AM1A}.
\end{example}

\section{Further applications}
 We briefly conclude with some further applications of our acceleration method, 
 in addition to our hypergeometric accelerations that yield formulas of Ramanujan type. 
 For the sake of brevity, we leave the requisite details to the leader. 

 As a natural variant of the input functions involved in Section \ref{sectionyielding}, 
 we consider the use of bivariate rational function factors 
 for input functions to be applied via our iterative method, instead of 
 only using combinations of Pochhammer symbols, and this leads us to set 
 For example, by setting 
\begin{equation}\label{210250512534AM1A}
 F(n, k, a) := \frac{ (a+e)_{k+d} }{ (n+a + f)_{k+c} (n+2k+b)}, 
\end{equation}
 this leads us to a variety of accelerations, such as the following. 

\begin{example}
 Setting $(a, b, c, d, e, f, n) = 
 ( -\frac{1}{3}, 0, -\frac{1}{3}, -\frac{1}{3}, 0, 0, \frac{2}{3} )$ after applying the iterative acceleration
 method to \eqref{210250512534AM1A}, we obtain 
 the motivating $\pi$ formula highlighted in \eqref{20250428911PM1A}. 
\end{example}

 By setting the input function $F(n, k, a)$ so as to involve combinations of Pochhammer symbols
 so that $a$ and $n$ having mixed signs, as in the input function
 \begin{equation}\label{202505191018AM1A}
 F(n, k, a) := \frac{ (a-n)_{k+f} (b)_{k+e} }{ (n+d)_{k+a} (n)_{k+c} }, 
\end{equation}
 this has led us toward accelerations as in the following.

\begin{example}
 Setting $(a, b, c, d, e, f, n) = 
 ( \frac{1}{2}, \frac{1}{2}, \frac{1}{2}, 0, 0, 1, 1 )$ after applying the iterative acceleration
 method to \eqref{202505191018AM1A}, we obtain 
 the motivating $\pi$ formula highlighted in \eqref{20250401333A111111M1A}. 
\end{example}

 Similarly, by setting 
\begin{equation}\label{2022520252129210229AM2A}
 F(n, k, a) := \frac{ (-a)_{k+f} (b)_{k+e} }{ (n+a)_{k+d} (n)_{k+c} }, 
\end{equation}
 this has led us toward accelerations as in the following. 

\begin{example}
 Setting $(a, b, c, d, e, f, n) = 
 ( \frac{1}{2}, -\frac{1}{2}, 0, -\frac{1}{2}, 0, 0, 1 )$ after applying the iterative
 acceleration method to \eqref{2022520252129210229AM2A}, we obtain 
 the motivating $\pi$ formula highlighted in \eqref{20250429148AM1A}. 
\end{example}

 By experimenting with variants of input functions as in 
 \eqref{207275707424733PM1A} and \eqref{27072775707570797775777P7M71A}, 
 this has led us toward setting 
\begin{equation}\label{2082858085818910283403A28M2A}
 F(n, k, a) := \frac{ (a)_{k + f} (b)_{k + e} }{ (n+a)_{k+c} (n+a)_{k+d} }, 
\end{equation}
 and toward the accelerations as in the following. 

\begin{example}
 Setting $(a, b, c, d, e, f, n) = 
 ( \frac{2}{3}, \frac{1}{3}, -\frac{1}{3}, \frac{2}{3}, \frac{1}{3}, -\frac{1}{3}, \frac{2}{3} )$
 after applying the iterative acceleration method to \eqref{2082858085818910283403A28M2A}, we obtain 
 the motivating $\pi$ formula shown in \eqref{motiv27bating20250423526}. 
\end{example}

 As another variant of input functions as in 
 \eqref{207275707424733PM1A} and \eqref{27072775707570797775777P7M71A}, 
 we set 
\begin{equation}\label{20250444434543139313043P3A3M32AA}
 F(n, k, a) := \frac{ (d)_{k + f} (b)_{k + e} }{ (n+a)_{k} (n)_{k+c}}, 
\end{equation}
 and this leads us toward accelerations as in the following. 

\begin{example}
 Setting $(a, b, c, d, e, f, n) = 
 ( \frac{1}{2}, \frac{1}{2}, \frac{1}{2}, -\frac{1}{2}, 0, 0, \frac{1}{2} )$
 after applying the iterative acceleration method to \eqref{20250444434543139313043P3A3M32AA}, we obtain 
 the $\pi$ formula highlighted as a motivating result in \eqref{202504050130610000}. 
\end{example}

 We have experimentally discovered that 
 the application of our iterative method to 
\begin{equation}\label{27072757075717971707487A7M7A}
 F(n, k, a) : = (-1)^{k} \frac{ (a+e)_{k+c}^{2} }{ (n+a+d)_{k+b}^2 } 
\end{equation}
 yields accelerated series of convergence rate $\frac{4}{729}$, as below.

\begin{example}
 Setting $(a, b, c, d, e, f, n) = 
 ( \frac{1}{2}, \frac{1}{2}, \frac{1}{2}, \frac{1}{2}, 0, \frac{1}{2}, \frac{1}{2} )$ 
 after applying the iterative acceleration method to \eqref{27072757075717971707487A7M7A}, we obtain 
 the motivating $\pi$ formula highlighted in \eqref{20250651650146PM1A}. 
\end{example}



\end{document}